\documentclass[12pt]{amsart}
\usepackage{amssymb,amsmath,amsfonts,latexsym}
\usepackage{bm, enumerate}
\newcommand{\newsection}[1]{\setcounter{equation}{0} \section{#1}}
\newcommand{\bea}{\begin{eqnarray}}
\newcommand{\eea}{\end{eqnarray}}

\newcommand{\clb}{\mathcal{B}}

\newcommand{\cle}{\mathcal{E}}

\newcommand{\clh}{\mathcal{H}}

\newcommand{\clo}{\mathcal{O}}

\newcommand{\clw}{\mathcal{W}}

\newcommand{\D}{\mathbb{D}}

\newcommand{\raro}{\rightarrow}

\def\textmatrix#1&#2\\#3&#4\\{\bigl({#1 \atop #3}\ {#2 \atop #4}\bigr)}
\def\dispmatrix#1&#2\\#3&#4\\{\left({#1 \atop #3}\ {#2 \atop #4}\right)}
\newcommand{\be}{\begin{equation}}
\newcommand{\ee}{\end{equation}}
\newcommand{\ben}{\begin{eqnarray*}}
\newcommand{\een}{\end{eqnarray*}}

\newcommand{\NI}{\noindent}

\newcommand{\bi}{\begin{itemize}}
\newcommand{\ei}{\end{itemize}}
\newcommand{\diag}{\mbox{diag}}

\newcommand\tT{{\tilde{T}}}
\newcommand\h{\frac{1}{2}}

\newcommand\BL{\langle}
\newcommand\BR{\rangle}

\newtheorem{Theorem}{\sc Theorem}[section]
\newtheorem{Lemma}[Theorem]{\sc Lemma}
\newtheorem{Proposition}[Theorem]{\sc Proposition}
\newtheorem{Corollary}[Theorem]{\sc Corollary}
\newtheorem{Definition}[Theorem]{\sc Definition}
\newtheorem{Example}[Theorem]{\sc Example}
\newtheorem{Remark}[Theorem]{\sc Remark}

\newtheorem{Note}[Theorem]{\sc Note}
\newtheorem{Question}{\sc Question}
\newtheorem{ass}[Theorem]{\sc Assumption}
\newcommand{\bt}{\begin{Theorem}}
\def\beginlem{\begin{Lemma}}
\def\beginprop{\begin{Proposition}}
\def\begincor{\begin{Corollary}}
\def\begindef{\begin{Definition}}
\def\beginexamp{\begin{Example}}
\def\beginrem{\begin{Remark}}
\def\beginq{\begin{Question}}
\def\beginass{\begin{ass}}
\def\beginnote{\begin{Note}}
\newcommand{\et}{\end{Theorem}}
\def\endlem{\end{Lemma}}
\def\endprop{\end{Proposition}}
\def\endcor{\end{Corollary}}
\def\enddef{\end{Definition}}
\def\endexamp{\end{Example}}
\def\endrem{\end{Remark}}
\def\endq{\end{Question}}
\def\endass{\end{ass}}
\def\endnote{\end{Note}}

\begin{document}

\title[Tridiagonal kernels, left-invertible operators and Aluthge transforms]{Tridiagonal kernels and left-invertible operators with applications to Aluthge transforms}


\author[Das]{Susmita Das}
\address{Indian Statistical Institute, Statistics and Mathematics Unit, 8th Mile, Mysore Road, Bangalore, 560059,
India}
\email{susmita.das.puremath@gmail.com}

\author[Sarkar]{Jaydeb Sarkar}
\address{Indian Statistical Institute, Statistics and Mathematics Unit, 8th Mile, Mysore Road, Bangalore, 560059,
India}
\email{jay@isibang.ac.in, jaydeb@gmail.com}




\begin{abstract}
Given scalars $a_n (\neq 0)$ and $b_n$, $n \geq 0$, the tridiagonal kernel or band kernel with bandwidth $1$ is the positive definite kernel $k$ on the open unit disc $\mathbb{D}$ defined by
\[
k(z, w) = \sum_{n=0}^\infty \Big((a_n + b_n z)z^n\Big) \Big((\bar{a}_n + \bar{b}_n \bar{w}) \bar{w}^n \Big) \qquad (z, w \in \mathbb{D}).
\]
This defines a reproducing kernel Hilbert space $\mathcal{H}_k$ (known as tridiagonal space) of analytic functions on $\mathbb{D}$ with $\{(a_n + b_nz) z^n\}_{n=0}^\infty$ as an orthonormal basis. We consider shift operators $M_z$ on $\mathcal{H}_k$ and prove that $M_z$ is left-invertible if and only if $\{|{a_n}/{a_{n+1}}|\}_{n\geq 0}$ is bounded away from zero. We find that, unlike the case of weighted shifts, Shimorin's models for left-invertible operators fail to bring to the foreground the tridiagonal structure of shifts. In fact, the tridiagonal structure of a kernel $k$, as above, is preserved under Shimorin model if and only if $b_0=0$ or that $M_z$ is a weighted shift. We prove concrete classification results concerning invariance of tridiagonality of kernels, Shimorin models, and positive operators.

We also develop a computational approach to Aluthge transforms of shifts. Curiously, in contrast to direct kernel space techniques, often Shimorin models fails to yield tridiagonal Aluthge transforms of shifts defined on tridiagonal spaces.
\end{abstract}

\maketitle

\tableofcontents

\newsection{Introduction}

The theory of left-invertible weighted shifts or multiplication operators $M_z$ on ``diagonal'' reproducing kernel Hilbert spaces is one of the most useful in operator theory, function theory, and operator algebras (see the classic by Shields \cite{Shields}). Given a bounded sequence of positive real numbers $\alpha = \{\alpha_n\}_{n \geq 0}$, and an orthonormal basis $\{e_n\}_{n \geq 0}$ of an infinite-dimensional Hilbert space $\clh$ (complex separable), the operator $S_{\alpha}$ defined by
\begin{equation}\label{eqn: sec 1 S alpha}
S_{\alpha} e_n = \alpha_n e_{n+1} \quad \quad (n \geq 0),
\end{equation}
is called a \textit{weighted shift} with weights $\{\alpha_n\}_{n \geq 0}$. In this case, $S_{\alpha}$ is bounded ($S_{\alpha} \in \clb(\clh)$ in short) and $\|S_{\alpha}\| = \sup_n \alpha_{n}$. If the sequence $\{\alpha_n\}_{n\geq 0}$ is bounded away from zero, then $S_{\alpha}$ is a left-invertible but non-invertible operator. Note that the multiplication operator $M_z$ on (most of the) diagonal reproducing kernel Hilbert spaces is the function theoretic counterpart of left-invertible weighted shifts which includes the Dirichlet shift, the Hardy shift, and the weighted and unweighted Bergman shifts, etc.

The main focus of this article is to study shifts on the ``next best'' concrete analytic kernels, namely, tridiagonal kernels. This notion was introduced by Adams and McGuire  \cite{Adam 2001} in 2001 (also see the motivating paper by Adams, McGuire and Paulsen \cite{Paulsen 92}). However, in spite of its natural appearance and potential applications, far less attention has been paid to the use of tridiagonal kernels in the aforementioned subjects. On the other hand, Shimorin \cite{SS} developed the idea of analytic models of left-invertible operators at about the same time as Adams and McGuire, which has been put forth as a key model for left-invertible operators by a number of researchers \cite{CCP, GSP, Stochel, Pawel}.

In the present paper we consider the next level of shifts on tridiagonal spaces, namely left-invertible shifts on tridiagonal spaces. We also discuss the pending and inevitable comparisons between Shimorin's analytic models of left-invertible operators and Adams and McGuire's theory of left-invertible shifts on tridiagonal spaces. In particular (and curiously enough), we find that, unlike the case of weighted shifts, Shimorin models fail to bring to the foreground the tridiagonal structure of shifts. We resolve this dilemma by presenting a complete classification of tridiagonal kernels that are preserved under Shimorin models.

We also prove a number of results concerning left-invertible properties of shifts on tridiagonal spaces, new tridiagonal spaces from the old, classifications of quasinormal operators, rank-one perturbations of left inverses, a computational approach to Aluthge transforms of shifts, etc. Again, curiously enough, some of our definite computations in the setting of tridiagonal kernels verify that the direct reproducing kernel Hilbert space technique is somewhat more powerful than Shimorin models. We also provide a family of instructive examples and supporting counterexamples.

To demonstrate the main contribution of this paper, it is now necessary to disambiguate central concepts. Needless to say, the theory of reproducing kernel Hilbert spaces will play a central role in this paper. Briefly stated, the essential idea of reproducing kernel Hilbert space \cite{Aronszajn 1050} is to single out the role of positive definiteness of inner products, multipliers and bounded point evaluations of function Hilbert spaces. We denote by $\D = \{z \in \mathbb{C}: |z| < 1\}$ the open unit disc in $\mathbb{C}$. Let $\cle$ be a Hilbert space. A function $k : \D \times \D \raro \clb(\cle)$ is called an \textit{analytic kernel} if $k$ is positive definite, that is,
\[
\sum_{i,j=1}^n \langle k(z_i, z_j) \eta_j, \eta_i \rangle_{\cle} \geq 0,
\]
for all $\{z_i\}_{i=1}^n \subseteq \D$, $\{\eta_i\}_{i=1}^n \subseteq \cle$ and $n \in \mathbb{N}$, and $k$ analytic in the first variable. In this case there exists a Hilbert space $\clh_k$, which we call \textit{analytic reproducing kernel Hilbert space} (\textit{analytic Hilbert space}, in short), of $\cle$-valued analytic functions on $\D$ such that $\{k(\cdot, w) \eta : w \in \D, \eta \in \cle\}$ is a total set in $\clh_k$ with the \textit{reproducing} property $\langle f, k(\cdot, w) \eta \rangle_{\clh_k} = \langle f(w), \eta \rangle_{\cle}$ for all $f \in \clh_k$, $w \in \D$, and $\eta \in \cle$. The \textit{shift operator} on $\clh_k$ is the multiplication operator $M_z$ (which will be assumed to be bounded) defined by \
\[
(M_z f)(w) = w f(w) \qquad (f \in \clh_k, w \in \D).
\]
Note that there exist $C_{mn} \in \clb(\cle)$ such that $k(z, w) = \sum_{m,n=0}^\infty C_{mn} z^m \bar{w}^n$, $z, w \in \D$. We say that $\clh_k$ is a diagonal reproducing kernel Hilbert space (and $k$ is a \textit{diagonal kernel}) if $C_{mn} = 0$ for all $|m-n| \geq 1$. We say that $k$ is a \textit{tridiagonal kernel} (or \textit{band kernel with bandwidth $1$}) if
\begin{equation}\label{eqn: def of TD C mn=0}
C_{mn} = 0 \quad \quad (|m-n| \geq 2).
\end{equation}
In this case, we say that $\clh_k$ is a \textit{tridiagonal space}. Now let $\{a_n\}_{n\geq 0}$ and $\{b_n\}_{n\geq 0}$ be a sequences of scalars. \textsf{In this paper, we will always assume that $a_n \neq 0$, for all $n \geq 0$.} Set
\[
f_n(z) = (a_n + b_n z) z^n \quad \quad (n \geq 0).
\]
Assume that $\{f_n\}_{n\geq 0}$ is an orthonormal basis of an analytic Hilbert space $\clh_k$. Then $\clh_k$ is a tridiagonal space, as
the well known fact from the reproducing kernel theory implies that
\begin{equation}\label{eqn intro: k(z,w)}
k(z, w) = \sum_{n=0}^{\infty} f_n(z) \overline{f_n(w)} \quad \quad (z, w \in \D).
\end{equation}

We now turn to Shimorin's analytic model of left-invertible operators \cite{SS}, which says that if $T \in \clb(\clh)$ is left-invertible and analytic (that is, $\cap_{n=0}^\infty T^n \clh = \{0\}$), then there exists an analytic Hilbert space $\clh_k (\subseteq \clo(\D, \clw))$ such that $T$ and $M_z$ on $\clh_k$ are unitarily equivalent, where $\clw = \ker T^* = \clh \ominus T \clh$ is the \textit{wandering subspace} of $T$, and $\clo(\D, \clw)$ is the set of $\clw$-valued analytic functions on $\D$. The \textit{Shimorin kernel} $k$ is explicit (see \eqref{eqn: k_T(z,w)}) which involves the \textit{Shimorin left inverse}
\begin{equation}\label{eqn: shim left inv}
L_T = (T^* T)^{-1} T^*,
\end{equation}
of $T$. The representation of the Shimorin kernel is useful in studying wandering subspaces of invariant subspaces of weighted shifts \cite{SS 2003, SS}. See \cite[Chapter 6]{Haakan} and \cite{Richter} in the context of the wandering subspace problem, and \cite{Pawel} and the extensive list of references therein for recent developments and implementations of Shimorin models.

We prove the following set of results: In Section \ref{sect:TD spaces}, we present basic properties and constructions of tridiagonal spaces and Shimorin models. We introduce the core concept of this paper: An \textit{analytic tridiagonal kernel} is a scalar kernel $k$ as in \eqref{eqn intro: k(z,w)} such that $\mathbb{C}[z] \subseteq \clh_k$, and
\[
\sup_{n\geq 0} \Big| \frac{a_n}{a_{n+1}}\Big| < \infty \quad \mbox{and} \quad \limsup_{n\geq 0} \Big| \frac{b_n}{a_{n+1}}\Big| < 1,
\]
(which ensures that $M_z$ on $\clh_k$ is bounded) and $\{| \frac{a_n}{a_{n+1}}|\}_{n\geq 0}$ is bounded away from zero. An analytic Hilbert space is called \textit{analytic tridiagonal space} if the kernel function is an analytic tridiagonal kernel. In Proposition \ref{Prop:shimorin diagonal}, we prove that weighted shifts behave well under Shimorin's analytic models.

In Section \ref{sect:TD and left inv}, we prove that $\{| \frac{a_n}{a_{n+1}}|\}_{n\geq 0}$ is bounded away is equivalent to the fact that $M_z$ on $\clh_k$ is left-invertible (see  Theorems \ref{thm: L M_z = I} and \ref{thm: iff bounded away}). We compute representations of Shimorin left inverses of shifts on analytic tridiagonal spaces (see Proposition \ref{prop: L is bounded} and Theorem \ref{Prop: L Matrix}).

Section \ref{sect: S kernels and TDS} starts with Example \ref{example: TD RKHS fail}, which shows that Shimorin kernels do not necessarily preserve the tridiagonal structure of kernels. We are nevertheless able to prove in Theorem \ref{thm: shimorin classify} that it does for a kernel $k$ of the form \eqref{eqn intro: k(z,w)} if and only if $M_z$ on $\clh_k$ is a weighted shift or
\[
b_0=0.
\]

The main result of Section \ref{sect: Positive and TDK} classifies positive operators $P$ on a tridiagonal space $\clh_k$ such that $K(z, w) := \langle P k(\cdot, w), k(\cdot, z) \rangle_{\clh_k}$ defines a tridiagonal kernel on $\D$. More specifically, if
\[
P = \begin{bmatrix}
c_{00}& c_{01} & c_{02} & c_{03} & \dots
\\
\bar{c}_{01} & c_{11} & c_{12} & c_{13} & \ddots
\\
\bar{c}_{02} & \bar{c}_{12} & c_{22} & c_{23}  & \ddots
\\
\bar{c}_{03} & \bar{c}_{13} & \bar{c}_{23} & c_{33} &
\ddots
\\
\vdots & \vdots & \vdots & \ddots & \ddots
\\
\end{bmatrix},
\]
denote the matrix representation of $P$ with respect to the basis $\{(a_n + b_n z) z^n\}_{n \geq 0}$ of $\clh_k$, then the kernel $K$ is tridiagonal if and only if $c_{0n} = (-1)^{n-1} \frac{\bar{b}_1 \cdots \bar{b}_{n-1}}{\bar{a}_2 \cdots \bar{a}_n} c_{01}$, $n \geq 2$, and $c_{mn} = (-1)^{n-m-1} \frac{\bar{b}_{m+1} \cdots \bar{b}_{n-1}}{\bar{a}_{m+2} \cdots \bar{a}_n} c_{m, m+1}$ for all $1 \leq m \leq n-2$ (see Theorem \ref{thm: P kernel and Mz}).

Section \ref{sect: quasi normal} deals with quasinormal shifts. Suppose $M_z$ is non-normal on an analytic tridiagonal space $\clh_k$. Denote by $P_{\mathbb{C} f_0}$ the orthogonal projection of $\clh_k$ onto $\mathbb{C} f_0$. In Theorem \ref{thm:quasinormal}, we prove that $M_z$ is quasinormal if and only if there exists $r>0$ such that
\[
M_z^* M_z - M_z M_z^* = r P_{\mathbb{C} f_0}.
\]

In Section \ref{sect: AT of Mz}, we compute Aluthge transforms of shifts. The notion of Aluthge transforms was introduced by Aluthge \cite{Aluthge} in his study of $p$-hyponormal operators. Let $\clh$ be a Hilbert space, $T \in \clb(\clh)$, and let $T = U|T|$ be the polar decomposition of $T$. Here, and throughout this note, $|T| = (T^* T)^{\frac{1}{2}}$ and $U$ is the unique partial isometry such that $\ker U = \ker T$. The \textit{Aluthge transform} of $T$ is the bounded linear operator
\[
\tT = |T|^{\frac{1}{2}} U |T|^{\frac{1}{2}}.
\]
The Aluthge transform of $\tT$ turns $T$ into a more ``normal'' operator while keeping intact the basic spectral properties of $T$ \cite{Pearcy-2000}. Evidently, the main difficulty associated with $\tT$ is to compute or represent the positive part $|T|$. This is certainly not true for weighted shifts: Since $|S_{\alpha}| = \mbox{diag} (\alpha_0, \alpha_1, \alpha_2, \ldots )$ (cf. Proposition \ref{Prop:shimorin diagonal}), it follows that $\tilde{S_{\alpha}} = S_{\sqrt{\alpha}}$, where
\[
\sqrt{\alpha} := \{\sqrt{\alpha_0 \alpha_1}, \sqrt{\alpha_1 \alpha_2}, \ldots\}.
\]
Therefore, $\tilde{S}_{\alpha}$ is also a weighted shift, namely $S_{\sqrt{\alpha}}$. Here we consider the next natural step: computation of $\tilde{M}_z$, where $M_z$ is a left-invertible shift on some analytic Hilbert space $\clh_k$. We prove that $\tilde{M}_z$ is also a left-invertible shift on some analytic Hilbert space $\clh_{\tilde{k}}$. The kernel $\tilde{k}$ can be obtained either via Shimorin's model (see Theorem \ref{thm: left analytic model}), which we call the \textit{Shimorin-Aluthge kernel of $M_z$}, or by a direct approach (see  Theorem \ref{thm: revisit model}), which we call the \textit{standard Aluthge kernel of $M_z$}. In Theorem \ref{thm: rank one perturb}, we prove that if $\mathbb{C}[z] \subseteq \clh_k \subseteq \clo(\D)$, then $L_{M_z}$ and $L_{\tilde{M}_z}$ are similar up to the perturbation of an operator of rank at most one. Moreover, in this setting Shimorin-Aluthge kernels are somewhat more explicit (see Theorem \ref{thm: left analytic model version 2}).

In Section \ref{sec: truncated} we consider truncated spaces (subclass of analytic tridiagonal spaces) in order to pinpoint more definite results, instructive examples, and counterexamples. A truncated space of order $r (\geq 2)$ is an analytic tridiagonal space $\clh_k$ with $k$ as in \eqref{eqn intro: k(z,w)} such that
\[
b_n = 0 \quad \quad (n \neq 2, 3, \ldots, r).
\]
The computational advantage of a truncated space is that it annihilate a rank one operator (see \eqref{eqn:F = rank one 2}) associated with $L_{M_z}$ of the shift $M_z$. As a result, in this case we are able to prove a complete classification of tridiagonal Shimorin-Aluthge kernels of shifts. This is the content of Theorem \ref{thm: truncated}. Curiously, the classification criterion of Theorem \ref{thm: truncated} is also the classification criterion of tridiagonality of standard Aluthge kernels (see Corollary \ref{cor:truncated SK = SK}).

In Section \ref{sect: final}, we comment on the assumptions in the definition of truncated kernels. We point out, at the other extreme, if one consider a (non-truncated) tridiagonal kernel $k$ with
\[
b_0 = b_1 = 1 \mbox{~or~} b_0 =1,
\]
and all other $b_i$'s are equal to $0$, then the standard Aluthge kernel of $M_z$ is a tridiagonal but the Shimorin-Aluthge kernel of $M_z$ is not. This is the main content of Example \ref{example: truncated SK not SK}. We conclude the paper by two observations concerning tridiagonal structures of standard Aluthge kernels and kernels of the form $(z, w) \mapsto \langle |M_z|^{-2} k(\cdot, w), k(\cdot, z) \rangle$.

We remark that some of the observations outlined in Sections \ref{sect: AT of Mz} and \ref{sec: truncated} are based on several more general results that have an independent interest in broader operator theory and function theoretic contexts.

\section{Preparatory results and examples}\label{sect:TD spaces}

In this section, we set up some definitions, collect some known facts about tridiagonal reproducing kernel Hilbert spaces and Shimorin analytic models, and observe some auxiliary results which are needed throughout the paper. We also explain the idea of Shimorin with the example of diagonal kernels (or equivalently, weighted shifts).

We start with tridiagonal spaces. Here we avoid finer technicalities \cite{Adam 2001} and introduce only the necessary features of tridiagonal spaces. Let $\cle$ be a Hilbert space, $k$ be a $\clb(\cle)$-valued analytic kernel on $\D$, and let $\clh_k \subseteq \clo(\D, \cle)$ be the corresponding reproducing kernel Hilbert space. Then there exists a sequence $\{C_{mn}\}_{m,n \geq 0} \subseteq \clb(\cle)$ such that
\[
k(z, w) = \sum_{m,n=0}^\infty C_{mn} z^m \bar{w}^n \quad \quad (z, w \in \D).
\]
Recall that (see \eqref{eqn: def of TD C mn=0}) $k$ is a \textit{tridiagonal kernel} if $C_{mn} = 0$, $|m-n| \geq 2$. We say that $\clh_k$ is a \textit{tridiagonal space} if $k$ is tridiagonal. We now single out two natural tridiagonal spaces.

\begin{Definition}
A tridiagonal space $\clh_k$  is called semi-analytic tridiagonal space if $\mathbb{C}[z] \subseteq \clh_k \subseteq \clo(\D)$, and there exist scalars $\{a_n\}_{n\geq 0}$ and $\{b_n\}_{n\geq 0}$, $a_n \neq 0$ for all $n \geq 0$, such that
\begin{equation}\label{eqn: sup of an bn}
\sup_{n\geq 0} \Big| \frac{a_n}{a_{n+1}}\Big| < \infty \quad \mbox{and} \quad \limsup_{n\geq 0} \Big| \frac{b_n}{a_{n+1}}\Big| < 1,
\end{equation}
and $\{f_n\}_{n\geq 0}$ is an orthonormal basis of $\clh_k$, where
\begin{equation}\label{eqn: f_n(z) sec 3}
f_n(z) = (a_n + b_n z)z^n \quad \quad (n \geq 0).
\end{equation}
\end{Definition}

Note that the conditions in \eqref{eqn: sup of an bn} ensure that the shift $M_z$ is a bounded linear operator on $\clh_k$ \cite[Theorem 5]{Adam 2001}. We refer the reader to \cite[Theorem 2]{Adam 2001} on the containment of polynomials.

\begin{Definition}\label{def: 2nd}
A semi-analytic tridiagonal space $\clh_k$ is said to be analytic tridiagonal space if the sequence $\{| \frac{a_n}{a_{n+1}}|\}_{n\geq 0}$ is bounded away from zero, that is, there exists $\epsilon>0$ such that
\begin{equation}\label{eqn: bounded away an/bn}
\Big| \frac{a_n}{a_{n+1}}\Big| > \epsilon \quad \quad (n \geq 0).
\end{equation}
\end{Definition}
A scalar kernel $k$ is called \textit{semi-analytic (analytic) tridiagonal kernel} if the corresponding reproducing kernel Hilbert space $\clh_k$ is a semi-analytic (an analytic) tridiagonal space.

It is important to note that \eqref{eqn: bounded away an/bn} is essential for left invertibility of $M_z$. As we will see in Theorem \ref{thm: iff bounded away}, if $\clh_k (\supseteq \mathbb{C}[z])$ is a tridiagonal space corresponding to the orthonormal basis $\{f_n\}_{n\geq 0}$ as in \eqref{eqn: f_n(z) sec 3}, and if $\{a_n\}_{n\geq 0}$ and $\{b_n\}_{n\geq 0}$ satisfies the conditions in \eqref{eqn: sup of an bn}, then condition \eqref{eqn: bounded away an/bn} is equivalent to the left invertibility of  $M_z$ on $\clh_k$. Also recall that the weighted shift $S_{\alpha}$ with weights $\{\alpha_n\}_{n \geq 0}$ (see \eqref{eqn: sec 1 S alpha}) is bounded if and only if $\sup_{n\geq0} \alpha_n < \infty$. In this case, $S_{\alpha}$ is left-invertible if and only if $\{\alpha_n\}_{n\geq 0}$ is bounded away from zero (cf. Proposition \ref{Prop:shimorin diagonal}). By translating this into the setting of analytic Hilbert spaces \cite[Proposition 7]{Shields}, it is clear that the conditions in Definition \ref{def: 2nd} are natural. For instance, if $b_n = 0$, $n \geq 0$, then \eqref{eqn: bounded away an/bn} is a necessary and sufficient condition for left invertibility of shifts on diagonal kernels.

Suppose $k$ is a semi-analytic tridiagonal kernel. Note that $k(z, w) = \sum_{n=0}^{\infty} f_n(z) \overline{f_n(w)}$ (see \eqref{eqn intro: k(z,w)}). Now fix $n \geq 0$, and write $z^n = \sum_{m=0}^\infty \alpha_m f_m$ for some $\alpha_m \in \mathbb{C}$, $m\geq 0$. Then
\[
z^n = \alpha_0 a_0 + \sum_{m=1}^{\infty} (\alpha_{m-1} b_{m-1} + \alpha_m a_m) z^m.
\]
Thus comparing coefficients, we have $\alpha_0 = \alpha_1 = \cdots = \alpha_{n-1} = 0$, and $\alpha_n = \frac{1}{a_n}$, as $a_i$'s are non-zero scalars. Since $\alpha_{n+j-1} b_{n+j-1} + \alpha_{n+j} a_{n+j} = 0$, it follows that $\alpha_{n+j} = - \frac{\alpha_{n+j-1} b_{n+j-1}}{a_{n+j}}$, and thus $\alpha_{n+j} = \frac{(-1)^j}{a_n} \frac{ b_n b_{n+1} \cdots b_{n+j-1}}{a_{n+1} \cdots a_{n+j}}$ for all $j \geq 1$. This implies
\begin{equation}\label{eq: z^n formula}
z^n = \frac{1}{a_n} \sum_{m=0}^\infty (-1)^m \Big(\frac{\prod_{j=0}^{m-1} b_{n+j}}{\prod_{j=0}^{m-1} a_{n+j+1}}\Big) f_{n+m} \quad \quad (n \geq 0),
\end{equation}
where $\prod_{j=0}^{-1} x_{n+j} := 1$. With this, we now proceed to compute $M_z$ \cite[Section 3]{Adam 2001}. Let $n \geq 0$. Then $M_z f_n = a_n z^{n+1} + b_n z^{n+2}$ implies that
\[
M_z f_n = \frac{a_n}{a_{n+1}} f_{n+1} + (b_n - \frac{a_n b_{n+1}}{a_{n+1}}) z^{n+2} = \frac{a_n}{a_{n+1}} f_{n+1} + a_{n+2} (\frac{b_n}{a_{n+2}} - \frac{a_n}{a_{n+1}} \frac{b_{n+1}}{a_{n+2}}) z^{n+2},
\]
that is
\begin{equation}\label{eq: M_z f_n formula}
M_z f_n  = \frac{a_n}{a_{n+1}} f_{n+1} + a_{n+2} c_n z^{n+2},
\end{equation}
where
\begin{equation}\label{eq: c_n b_n a_n}
c_n = \frac{a_n}{a_{n+2}} \Big(\frac{b_n}{a_{n}} - \frac{b_{n+1}}{a_{n+1}}\Big) \quad \quad (n \geq 0).
\end{equation}
Then \eqref{eq: z^n formula} implies that
\begin{equation}\label{eq: M_z}
M_z f_n  = \Big(\frac{a_n}{a_{n+1}}\Big) f_{n+1} + c_n \sum_{m=0}^\infty (-1)^m \Big(\frac{\prod_{j=0}^{m-1} b_{n+2+j}}{\prod_{j=1}^{m-1} a_{n+3+j}}\Big) f_{n+2+m} \quad \quad (n\geq 0),
\end{equation}
and hence, with respect to the orthonormal basis $\{f_n\}_{n \geq 0}$, we have (also see \cite[Page 729]{Adam 2001})
\begin{equation}\label{eqn: Mz matrix}
[M_z] = \begin{bmatrix}
0& 0 & 0 & 0 & \dots
\\
\frac{a_0}{a_1} & 0 & 0 & 0 & \ddots
\\
{c_0} & \frac{a_1}{a_2} & 0 & 0 & \ddots
\\
\frac{-c_0 b_2}{a_3} & c_1 & \frac{a_2}{a_3} & 0 & \ddots
\\
\frac{c_0b_2b_3}{a_3a_4} &\frac{-c_1b_3}{a_4} & c_2 & \frac{a_3}{a_4} & \ddots
\\
\frac{-c_0b_2b_3b_4}{a_3a_4a_5} &\frac{c_1b_3b_4}{a_4a_5} & \frac{-c_2b_4}{a_5} & c_3 & \ddots
\\
\vdots & \vdots & \vdots&\ddots &\ddots
\end{bmatrix}.
\end{equation}
The matrix representation of the conjugate of $M_z$ is going to be useful in what follows:

\begin{equation}\label{eqn: Mz^* matrix}
[M^*_z] = \begin{bmatrix}
0& \frac{\bar{a}_0}{\bar{a}_1} & \bar{c}_0 & \frac{-\bar{c}_0 \bar{b}_2}{\bar{a}_3} & \frac{-\bar{c}_0 \bar{b}_2 \bar{b}_3}{\bar{a}_3 \bar{a}_4}  & \dots
\\
0& 0& \frac{\bar{a}_1}{\bar{a}_2} & \bar{c}_1 & \frac{-\bar{c}_1 \bar{b}_3}{\bar{a}_4} & \ddots
\\
0 & 0 & 0 & \frac{\bar{a}_2}{\bar{a}_3} & \bar{c}_2 & \ddots
\\
0& 0 & 0 & 0 & \frac{\bar{a}_3}{\bar{a}_4} & \ddots
\\
\vdots & \vdots & \vdots&\vdots &\ddots &\ddots
\end{bmatrix}.
\end{equation}
In particular, $M_z$ is a weighted shift if and only if $c_n = 0$ for all $n \geq 0$. Also, by \eqref{eq: c_n b_n a_n}, we have $c_n=0$ if and only if $\frac{b_{n+1}}{a_{n+1}} = \frac{b_n}{a_n}$, $n \geq 0$.
Therefore, we have the following observation:

\begin{Lemma}\label{lemma: M_z weighted shift}
The shift $M_z$ on a semi-analytic tridiagonal space $\clh_k$ is a weighted shift if and only if $c_n = 0$ for all $n \geq 0$, or, equivalently, $\{\frac{b_n}{a_n}\}_{n \geq 0}$ is a constant sequence.
\end{Lemma}

The proof of the following lemma uses the assumption that $\mathbb{C}[z] \subseteq \clh_k$.

\begin{Lemma}\label{lemma: ker Mz*}
If $\clh_k$ is a semi-analytic tridiagonal space, then $\ker M_z^* = \mathbb{C} f_0$.
\end{Lemma}
\begin{proof}
Clearly, \eqref{eqn: Mz^* matrix} implies that $f_0 \in \ker M_z^*$. On the other hand, from $\mathbb{C}[z] \subseteq \clh_k$ we deduce that $f_n = M_z (a_n z^{n-1} + b_n z^n) \in \mbox{ran} M_z$ for all $n \geq 1$, and hence $\mbox{span} \{f_n: n \geq 1\} \subseteq \mbox{ran} M_z$. The result now follows from the fact that $\mathbb{C} f_0 = (\mbox{span} \{f_n: n \geq 1\})^\perp \supseteq \ker M_z^*$.
\end{proof}

Now we briefly describe the construction of Shimorin's analytic models of left-invertible operators. Let $\clh$ be a Hilbert space, and let $T \in \clb(\clh)$. We say that $T$ is \textit{left-invertible} if there exists $X \in \clb(\clh)$ such that $XT = I_{\clh}$. It is easy to check that this equivalently means that $T$ is bounded below, which is also equivalent to the invertibility of $T^* T$. Following Shimorin, a bounded linear operator $X \in \clb(\clh)$ is \textit{analytic} if
\begin{equation}\label{eqn: cap Xn H =0}
\bigcap_{n=0}^\infty X^n \clh = \{0\}
\end{equation}
Note that from the viewpoint of analytic Hilbert spaces, shifts are always analytic. Indeed, let $\clh_k \subseteq \clo(\Omega, \cle)$, where $\Omega \subseteq \mathbb{C}$ is a domain, and suppose the shift $M_z$ is bounded on $\clh_k$. If $f \in \bigcap_{n=0}^\infty M_z^n \clh_k$, then for each $n \geq 0$, there exists $g_n \in \clh_k$ such that $f = z^n g_n$. Since $\Omega$ is a domain and $f$ is analytic on $\Omega$, we see that $f \equiv 0$, that is, $\bigcap_{n=0}^\infty M_z^n \clh_k = \{0\}$.

Now let $T \in \clb(\clh)$ be a bounded below operator. We call $L_T := (T^* T)^{-1} T^*$ the \textit{Shimorin left inverse}, to distinguish it from other left inverses of $T$ (see \eqref{eqn: shim left inv}). Set
\[
\clw = \ker T^* = \clh \ominus T \clh,
\]
and $\Omega = \{z \in \mathbb{C}: |z| < \frac{1}{r(L_T)}\}$, where $r(L_T)$ is the spectral radius of $L_T$. Then
\begin{equation}\label{eqn: k_T(z,w)}
k_T(z, w) = P_{\clw} (I - z L_T)^{-1} (I - \bar{w} L_T^*)^{-1}|_{\clw} \quad \quad (z, w \in \Omega),
\end{equation}
defines a $\clb(\clw)$-valued analytic kernel $k_T : \Omega \times \Omega \raro \clb(\clw)$, which we call the \textit{Shimorin kernel} of $T$ (see \cite[Corollary 2.14]{SS}). We lose no generality by assuming, as we shall do, that $\Omega = \D$. If, in addition, $T$ is analytic, then the unitary $U : \clh \raro \clh_k$ defined by
\begin{equation}\label{eqn: sec 2 Unitary}
(Uf)(z) = \sum_{n=0}^\infty (P_{\clw} L_T^n f) z^n \quad \quad (f \in \clh, z \in \D),
\end{equation}
satisfies $U T = M_z U$ \cite{SS}. More precisely, we have the following result:

\begin{Theorem}\label{thm-Shimorin}
Let $T \in \clb(\clh)$ be an analytic left-invertible operator. Then $T$ on $\clh$ and $M_z$ on $\clh_{k_T}$ are unitarily equivalent.
\end{Theorem}

Denote by $P_{\clw}$ the orthogonal projection of $\clh$ onto $\clw = \ker T^*$. It follows that
\begin{equation}\label{eqn:P_W = I- TL}
P_{\clw} = I_{\clh} - TL_T,
\end{equation}
This plays an important role (in the sense of Wold decomposition of left-invertible operators) in the proof of the above theorem. The following equality will be very useful in what follows.

\begin{Lemma}\label{prop: L tilde T}
If $T$ is a left-invertible operator on $\clh$, then $L_T L_T^* = |T|^{-2}$.
\end{Lemma}
\begin{proof} This follows from the fact that $L_T L_T^* = (T^* T)^{-1} T^* T (T^* T)^{-1} = (T^* T)^{-1}$.
\end{proof}

In the case of left-invertible weighted shifts $S_{\alpha}$ (see \eqref{eqn: sec 1 S alpha}), it is perhaps known that the shift $M_z$ on $\clh_{k_{S_{\alpha}}}$ corresponding to the Shimorin kernel $k_{S_{\alpha}}$ is also a weighted shift. Nonetheless, we sketch the proof here for the sake of completeness.

\begin{Proposition}\label{Prop:shimorin diagonal}
Let $S_{\alpha}$ be the weighted shift with weights $\{\alpha_n\}_{n\geq 0}$. If $\{\alpha_n\}_{n\geq 0}$ is bounded away from zero, then $S_{\alpha}$ is left-invertible, and the Shimorin kernel $k_{S_{\alpha}}$ is diagonal.
\end{Proposition}
\begin{proof}
Let $\{e_n\}_{n\geq 0}$ be an orthonormal basis of a Hilbert space $\clh$, and let $S_{\alpha} e_n = \alpha_n e_{n+1}$ for all $n \geq 0$. Observe that $S_{\alpha}^* e_n = \alpha_{n-1} e_{n-1}$, $n \geq 1$, and $S_{\alpha}^* e_0 = 0$. Then $\clw = \ker S_{\alpha}^* = \mathbb{C} e_0$, and $S_{\alpha}^* S_{\alpha} e_n = \alpha_n^2 e_n$ for all $n \geq 0$. Since $S_{\alpha}^* S_{\alpha}$ is a diagonal operator and $\{\alpha_n\}_{n\geq 0}$ is bounded away from zero, it follows that $S_{\alpha}^* S_{\alpha}$ is invertible, and hence $S_{\alpha}$ is left-invertible. Then the Shimorin left inverse $L_{S_{\alpha}} := (S_{\alpha}^* S_{\alpha})^{-1} S_{\alpha}^*$ is given by
\begin{equation}\label{eqn: sec 2 L S en}
L_{S_{\alpha}} e_n =
\begin{cases}
0 & \mbox{if } n =0
\\
\frac{1}{\alpha_{n-1}} e_{n-1}  & \mbox{if } n \geq 1.
\end{cases}
\end{equation}
Therefore, $L_{S_{\alpha}}$ is the backward shift, and
\begin{equation}\label{eqn: sec 2 L m S en}
L_{S_{\alpha}}^m e_n =
\begin{cases}
0 & \mbox{if } m > n
\\
\frac{1}{\alpha_{0} \cdots \alpha_{n-1}} e_{0}  & \mbox{if } m=n
\\
\frac{1}{\alpha_{n-1} \cdots \alpha_{n-m}} e_{n-m}  & \mbox{if } m <n,
\end{cases}
\end{equation}
for all $m \geq 1$. Moreover, $L_{S_{\alpha}}^{*m} e_n = \frac{1}{\alpha_n \alpha_{n+1} \cdots \alpha_{n+m-1}} e_{n+m}$ for all $n \geq 0$ and $m \geq 1$. In particular, $L_{S_{\alpha}}^{*m} e_0 = \frac{1}{\alpha_0 \alpha_{1} \cdots \alpha_{m-1}} e_{m}$, $m\geq 1$, and thus, for each $(m, n) \neq (0,0)$, we have clearly
\[
P_{\clw} L_{S_{\alpha}}^{m} L_{S_{\alpha}}^{*n} e_0 =
\begin{cases}
0 & \mbox{if } m \neq n
\\
\frac{1}{(\alpha_{0} \cdots \alpha_{n-1})^2} e_{0}  & \mbox{if } m = n.
\end{cases}
\]
This immediately gives $k_{S_{\alpha}}(z,w) = \sum_{n=0}^\infty (P_{\clw} L_{S_{\alpha}}^n L_{S_{\alpha}}^{*n}|_{\clw}) (z \bar{w})^n$ for all $z, w \in \D$, where $\clw = \mathbb{C} e_0$. In particular, the Shimorin kernel $k_{S_{\alpha}}$ is a diagonal kernel. Finally, identifying $\clw$ with $\mathbb{C}$ and setting $\beta_n = \frac{1}{\alpha_{0} \cdots \alpha_{n-1}}$, $n\geq 1$, we get
\[
k_{S_{\alpha}}(z,w) = 1 +  \sum_{n=1}^\infty \frac{1}{\beta_n^2} (z \bar{w})^n \quad \quad (z, w \in \D).
\]
\end{proof}

Notice in the above, the Shimorin left inverse $L_{S_{\alpha}}$ is the backward shift corresponding to the weight sequence $\{\frac{1}{\alpha_n}\}_{n \geq 0}$, that is,
\[
L_{S_{\alpha}} =  \begin{bmatrix}
0& \frac{1}{\alpha_0} & 0 & 0 & \dots
\\
0 & 0 &  \frac{1}{\alpha_1} & 0 & \ddots
\\
0 &0&0 &  \frac{1}{\alpha_2} & \ddots
\\
0 &0&0 & 0 & \ddots
\\
\vdots & \vdots & \vdots & \ddots & \ddots
\\
\end{bmatrix}.
\]

In the setting of Proposition \ref{Prop:shimorin diagonal}, we now turn to the unitary map $U: \clh \raro \clh_{k_{S_{\alpha}}}$, where $\clh_{k_{S_{\alpha}}} \subseteq \clo(\D, \clw)$, and $(U f)(z) = \sum_{n=0}^\infty (P_{\clw} L_{S_{\alpha}}^n f) z^n$ for all $f \in \clh$ and $z \in \D$ (see \eqref{eqn: sec 2 Unitary}). Set $f_n = U e_n$, $n \geq 0$. Since $\clw = \mathbb{C} e_0$, \eqref{eqn: sec 2 L S en} yields $f_0 = U e_0 = P_{\clw} e_0 = e_0$. On the other hand, if $n \geq 1$, then \eqref{eqn: sec 2 L m S en} implies that
\[
P_{\clw} L_{S_{\alpha}}^{m} e_n =
\begin{cases}
\frac{1}{\beta_n} e_0  & \mbox{if } m = n
\\
0 & \mbox{otherwise,}
\end{cases}
\]
and hence $f_n = \frac{1}{\beta_n} z^n e_0$. Therefore $\{e_0\} \cup \{\frac{1}{\beta_n} z^n e_0\}_{n\geq 1}$ is the orthonormal basis of $\clh_{k_{S_{\alpha}}}$ corresponding to $U$. Moreover, for each $n \geq 1$, we have
\[
\begin{split}
M_z (\frac{1}{\beta_n} z^n e_0) = \frac{1}{\beta_n} z^{n+1}e_0 = \alpha_n \frac{1}{\beta_{n+1}} z^{n+1}e_0= \alpha_n (\frac{1}{\beta_{n+1}} z^{n+1}e_0),
\end{split}
\]
and hence $M_z$ on $\clh_{k_{S_{\alpha}}}$ is also a weighted shift with the same weights $\{\alpha_n\}_{n\geq 0}$.

\section{Tridiagonal spaces and left-invertibility}\label{sect:TD and left inv}

The main contribution of this section is the left invertibility and representations of Shimorin left inverses of shifts on tridiagonal reproducing kernel Hilbert spaces. Recall that the conditions in \eqref{eqn: sup of an bn} ensures that the shift $M_z$ is bounded on the semi-analytic tridiagonal space $\clh_k$. Here we use the remaining condition \eqref{eqn: bounded away an/bn} to prove that $M_z$ is left-invertible.

Before we state and prove the result, we need to construct a specific bounded linear operator. The choice of this operator is not accidental, as we will see in Theorem \ref{Prop: L Matrix} that it is nothing but the Shimorin left inverse of $M_z$. For each $n\geq 1$, set
\begin{equation}\label{eqn:d_n}
d_n = \frac{b_n}{a_n} - \frac{b_{n-1}}{a_{n-1}}.
\end{equation}

\begin{Proposition}\label{prop: L is bounded}
Let $k$ be an analytic tridiagonal kernel corresponding to the orthonormal basis $\{f_n\}_{n\geq 0}$, where $f_n(z) = (a_n + b_z z) z^n$, $n \geq 0$. Then the linear operator $L$ represented by
\[
[L] = \begin{bmatrix}
0 & \frac{a_1}{a_0} & 0 & 0  & 0 & \dots
\\
0 & {d_1} & \frac{a_2}{a_1} & 0  & 0 & \ddots
\\
0 & \frac{-d_1b_1}{a_2} & d_2 & \frac{a_3}{a_2} & 0 & \ddots
\\
0 & \frac{d_1b_1b_2}{a_2a_3} &\frac{-d_2b_2}{a_3} & d_3 &\frac{a_4}{a_3} & \ddots
\\
0 & \frac{-d_1b_1b_2b_3}{a_2a_3a_4} &\frac{d_2b_2b_3}{a_3a_4} & \frac{-d_3b_3}{a_4} & d_4 & \ddots
\\
\vdots & \vdots & \vdots&\vdots &\ddots &\ddots
\end{bmatrix},
\]
with respect to the orthonormal basis $\{f_n\}_{n \geq 0}$ defines a bounded linear operator on $\clh_k$.
\end{Proposition}
\begin{proof}
For each $n \geq 1$, we have clearly $d_n=\frac{b_n}{a_n}-\frac{b_{n-1}}{a_{n-1}} = \frac{a_{n+1}}{a_n}\frac{b_n}{a_{n+1}}-\frac{a_n}{a_{n-1}}\frac{b_{n-1}}{a_n}$, and hence
\[
|d_n| \leq  \Big|\frac{a_{n+1}}{a_n}\Big| \Big| \frac{b_n}{a_{n+1}} \Big|+ \Big| \frac{a_n}{a_{n-1}}\Big| \Big| \frac{b_{n-1}}{a_n} \Big|.
\]
Since $\{|\frac{a_n}{a_{n+1}}|\}_{n\geq 0}$ is bounded away from zero (see \eqref{eqn: bounded away an/bn}), we have that $\sup_{n \geq 0}| \frac{a_{n+1}}{a_{n}}| < \infty$. This and the second assumption then imply that $\{d_n\}$ is a bounded sequence.

\NI Let $S$ denote the matrix obtained from $[L]$ by deleting all but the superdiagonal elements of $[L]$. Similarly, $L_0$ denote the matrix obtained from $[L]$ by deleting all but the diagonal elements of $[L]$, and in general, assume that $L_i$ denote the matrix obtained from $[L]$ by deleting all but the $i$-th subdiagonal of $[L]$, $i=0, 1, 2 \ldots$. Since
\[
L = S + \sum_{i\geq 0} L_i,
\]
it clearly suffices to prove that $S$ and $\{L_i\}_{i \geq 0}$ are bounded, and $S + \sum_{i\geq 0} L_i$ is absolutely convergent. Note that $\|S\| = \sup_{n \geq 0} | \frac{a_{n+1}}{a_{n}}| < \infty$. Moreover, our assumption $\limsup_{n\geq 0}| \frac{b_n}{a_{n+1}}| < 1$ implies that there exist $r < 1$ and $n_0 \in \mathbb{N}$ such that
\[
\Big| \frac{b_n}{a_{n+1}} \Big| < r \qquad (n \geq n_0).
\]
Set
\[
M = \sup_{n \geq 1} \Big\{\Big|\frac{b_n}{a_{n+1}}\Big|, |d_n| \Big\}.
\]
Then $\|L_i\| \leq M^{i+1}$ for all $i=0, \ldots, n_0$, and
\[
\|L_i\| \leq M^{n_0+1} r^{i-n_0} \quad \quad (i > n_0),
\]
from which it follows that
\[
\begin{split}
\|S\| + \sum_{i\geq 0} \|L_i\| & = \sup_{n \geq 0} \Big| \frac{a_{n+1}}{a_{n}}\Big| + \sum_{0 \leq i\leq n_0} \|L_i\| + \sum_{i\geq n_0+1} \|L_i\|
\\
& \leq \sup_{n \geq 0} \Big| \frac{a_{n+1}}{a_{n}}\Big| + \sum_{0 \leq i\leq n_0} \|L_i\| + M^{n_0+1} \Big(\sum_{i \geq n_0+1} r^{i-n_0} \Big)
\\
& \leq \sup_{n \geq 0} \Big| \frac{a_{n+1}}{a_{n}}\Big| + \sum_{0 \leq i\leq n_0} \|L_i\| + M^{n_0+1} \frac{r}{1-r},
\end{split}
\]
and completes the proof of the theorem.
\end{proof}

We are now ready to prove that $M_z$ is left-invertible.

\begin{Theorem}\label{thm: L M_z = I}
In the setting of Proposition \ref{prop: L is bounded}, we have $L M_z = I_{\clh_k}$.
\end{Theorem}
\begin{proof}
We consider the matrix representations of $M_z$ and $L$ as in \eqref{eqn: Mz matrix} and Proposition \ref{prop: L is bounded}, respectively. Let $[L] [M_z] = (\alpha_{mn})_{m, n \geq0}$. Clearly it suffices to prove that $\alpha_{mn} = \delta_{mn}$. It is easy to see that $\alpha _{m,m+k}=0$ for all $k \geq 1$. Now by \eqref{eq: c_n b_n a_n}, we have
\begin{equation}\label{eqn: cn=ac dn}
c_n = - \frac{a_n}{a_{n+2}} d_{n+1} \quad \quad (n \geq 0).
\end{equation}
Note that the $n$-th column, $n \geq 0$, of $[M_z]$ is the transpose of
\[
\begin{split}
\Big(\underbrace{0, \ldots, 0}_{n+1}, \frac{a_n}{a_{n+1}}, c_n, - \frac{c_n b_{n+2}}{a_{n+3}},\ldots , (-1)^{m-n-2} \frac{c_n b_{n+2}\cdots b_{m-1}}{a_{n+3}\cdots a_m}, (-1)^{m-n-1} \frac{c_n b_{n+2}\cdots b_m}{a_{n+3}\cdots a_{m+1}}, \ldots \Big),
\end{split}
\]
and the $m$-th row, $m\geq 0$, of $[L]$ is given by
\[
\begin{split}
\Big(0\, , & \, (-1)^{m-1} \frac{ d_1 b_1\cdots b_{m-1}}{a_2\cdots a_m}, (-1)^{m-2} \frac{d_2 b_2\cdots b_{m-1}}{a_3\cdots a_m}, (-1)^{m-3}  \frac{d_3 b_3\cdots b_{m-1}}{a_4\cdots a_m}, \ldots
\\
& \quad \quad \quad \ldots, \frac{-d_{m-1}b_{m-1}}{a_m},d_m,\frac{a_{m+1}}{a_m},0,0,\ldots \Big).
\end{split}
\]
Now, if $n\leq (m-2)$, then the $\alpha_{mn}$ (the $(m, n)$-th entry of $[L] [M_z]$) is given by
\[
\begin{split}
\alpha_{mn} & = (-1)^{m-n-1}\frac{d_{n+1}b_{n+1}\cdots b_{m-1}}{a_{n+2}\cdots a_m} \, \frac{a_n}{a_{n+1}}+(-1)^{m-n-2}\frac{d_{n+2}b_{n+2}\cdots b_{m-1}}{a_{n+3}\cdots a_m}c_n
\\
& \quad +(-1)^{m-n-3}\frac{d_{n+3}b_{n+3}\cdots b_{m-1}}{a_{n+4}\cdots a_m}(-c_n\frac{b_{n+2}}{a_{n+3}})+\cdots+(-\frac{d_{m-1}b_{m-1}}{a_m})(-1)^{m-n-3} \times
\\
& \quad
c_n\frac{b_{n+2}\cdots b_{m-2}}{a_{n+3}\cdots a_{m-1}} + d_m(-1)^{m-n-2}c_n\frac{b_{n+2}\cdots b_{m-1}}{a_{n+3}\cdots a_m}+\frac{a_{m+1}}{a_m}(-1)^{m-n-1}c_n\frac{b_{n+2}\cdots b_m}{a_{n+3}\cdots a_m a_{m+1}},
\end{split}
\]
and hence, using \eqref{eqn: cn=ac dn}, we obtain
\[
\begin{split}
\alpha_{mn} & = (-1)^{m-n-1}d_{n+1}\frac{a_n b_{n+1}\cdots b_{m-1}}{a_{n+1}a_{n+2}\cdots a_m}+(-1)^{m-n-2}(-\frac{a_n}{a_{n+2}}d_{n+1})\frac{d_{n+2}b_{n+2}\cdots b_{m-1}}{a_{n+3}\cdots a_m} +
\\
& \quad (-1)^{m-n-2}(-\frac{a_n}{a_{n+2}}d_{n+1})(\frac{b_{n+2}}{a_{n+3}})(\frac{d_{n+3}b_{n+3}\cdots b_{m-1}}{a_{n+4}\cdots a_m})+ \cdots +
\\
&  \quad  \cdots + (-1)^{m-n-2}(-\frac{a_n}{a_{n+2}}d_{n+1}) \frac{d_{m-1} b_{n+2}\cdots b_{m-1}}{a_{n+3}\cdots a_m} +
\\
& \quad (-1)^{m-n-2}(-\frac{a_n}{a_{n+2}}d_{n+1})\frac{d_m b_{n+2}\cdots b_{m-1}}{a_{n+3}\cdots a_m} + (-1)^{m-n-1}(-\frac{a_n}{a_{n+2}}d_{n+1})(\frac{b_{n+2}\cdots b_m}{a_{n+3}\cdots a_m^2})
\\
& = (-1)^{m-n-1}d_{n+1} \Big(\frac{a_n b_{n+1}\cdots b_{m-1}}{a_{n+1}a_{n+2}\cdots a_m} + \frac{a_n b_{n+2}\cdots b_{m-1}}{a_{n+2}a_{n+3}\cdots a_m} d_{n+2} + \frac{a_n b_{n+2}\cdots b_{m-1}}{a_{n+2}a_{n+3}\cdots a_m} d_{n+3} +
\\
& \quad \cdots  + \frac{a_n b_{n+2}\cdots b_{m-1}}{a_{n+2}a_{n+3}\cdots a_m} d_{m-1}
+ \frac{a_n b_{n+2}\cdots b_{m-1}}{a_{n+2}a_{n+3}\cdots a_m} d_m - \frac{a_n b_{n+2}\cdots b_m}{a_{n+2}a_{n+3}\cdots a_m^2} \Big)
\\
& =(-1)^{m-n-1}d_{n+1} \frac{a_n b_{n+2}\cdots b_{m-1}}{a_{n+2}a_{n+3}\cdots a_m} \left(\frac{b_{n+1}}{a_{n+1}}+(d_{n+2}+d_{n+3}+\cdots +d_{m-1}+d_m)-\frac{b_m}{a_m}\right).
\end{split}
\]
Recall from \eqref{eqn:d_n} that $d_n = \frac{b_n}{a_n} - \frac{b_{n-1}}{a_{n-1}}$, $n \geq 1$. Then
\[
\alpha_{mn} =(-1)^{m-n-1}d_{n+1} \frac{a_n b_{n+2} \cdots b_{m-1}}{a_{n+2}a_{n+3}\cdots a_m} \left((\frac{b_{n+1}}{a_{n+1}}-\frac{b_m}{a_m}) + (\frac{b_m}{a_m}-\frac{b_{n+1}}{a_{n+1}})\right)  = 0.
\]
For the case $n= m-1$, we have
\[
\alpha_{m, m-1} = d_m (\frac{a_{m-1}}{a_m}) + \frac{a_{m+1}}{a_m}(c_{m-1}) = (\frac{a_{m-1}}{a_m}) d_m + \frac{a_{m+1}}{a_m} (-\frac{a_{m-1}}{a_{m+1}}d_m) = 0,
\]
and finally, $\alpha_{mm} = (\frac{a_{m+1}}{a_m}) (\frac{a_m}{a_{m+1}}) = 1$ completes the proof.
\end{proof}

In view of Theorem \ref{thm: L M_z = I}, let us point out, in particular (see the discussion following \eqref{eqn: cap Xn H =0}), that shifts on analytic tridiagonal spaces are always analytic:

\begin{Proposition}
If $k$ is an analytic tridiagonal kernel, then $M_z$ is an analytic left-invertible operator on $\clh_k$.
\end{Proposition}

Now let $\clh_k$ be an analytic tridiagonal space. Our aim is to compute the Shimorin left inverse $L_{M_z} = (M_z^* M_z)^{-1} M_z^*$ of $M_z$ on $\clh_k$. What we prove in fact is that $L$ in Proposition \ref{prop: L is bounded} is the Shimorin left inverse of $M_z$. First note that 
\begin{equation}\label{eqn:L_Mz back shift}
L_{M_z} z^n = z^{n-1} \qquad (n \geq 1).
\end{equation}
Indeed, $L_{M_z} z^n = (M_z^* M_z)^{-1} M_z^* M_z z^{n-1} = (M_z^* M_z)^{-1} (M_z^* M_z) z^{n-1}$. Therefore, $L_{M_z}$ is the backward shift on $\clh_k$ (a well known fact about Shimorin left inverses). On the other hand, by Lemma \ref{lemma: ker Mz*} we have $L_{M_z} f_0 = (M_z^* M_z)^{-1} M_z^* f_0 = 0$, and hence $L_{M_z} f_0 = 0$, which in particular yields
\begin{equation}\label{eqn:L_1}
L_{M_z} 1 = - \frac{b_0}{a_0}.
\end{equation}
Let $n \geq 1$. Using \eqref{eqn:d_n}, we have $L_{M_z} f_n = L_{M_z}(a_n z^n + b_n z^{n+1})= a_n z^{n-1} + b_n z^n$, which implies
\[
L_{M_z} f_n = \frac{a_n}{a_{n-1}} (a_{n-1} z^{n-1} + b_{n-1} z^n) + (b_n - \frac{a_n b_{n-1}}{a_{n-1}})z^n = \frac{a_n}{a_{n-1}} f_{n-1} + d_n a_n z^n,
\]
and hence $L_{M_z} f_n = \frac{a_n}{a_{n-1}} f_{n-1} + d_n(a_n z^n + b_n z^{n+1}) - d_n b_n z^{n+1}$. By \eqref{eq: z^n formula}, we have
\[
L_{M_z} f_n = \frac{a_n}{a_{n-1}} f_{n-1} + d_n f_n - d_n \Big(\sum_{m=0}^{\infty} (-1)^m \frac{\prod_{j=0}^{m} b_{n+j}}{\prod_{j=0}^{m} a_{n+1+j}} f_{n+1+m} \Big).
\]
This is precisely the left inverse $L$ of $M_z$ in Proposition \ref{prop: L is bounded}. Whence the next statement:

\begin{Theorem}\label{Prop: L Matrix}
Let $\clh_k$ be an analytic tridiagonal space. If $L$ is as in Proposition \ref{prop: L is bounded}, then the Shimorin left inverse $L_{M_z}$ of $M_z$ is given by $L_{M_z} = L$. In particular, $L_{M_z} f_0 = 0$,
and
\[
L_{M_z} f_n = \frac{a_n}{a_{n-1}} f_{n-1} + d_n f_n - d_n \Big(\sum_{m=0}^{\infty} (-1)^m \frac{\prod_{j=0}^{m} b_{n+j}}{\prod_{j=0}^{m} a_{n+1+j}} f_{n+1+m} \Big) \quad \quad (n\geq 1),
\]
where $d_n = \frac{b_n}{a_n} - \frac{b_{n-1}}{a_{n-1}}$ for all $n \geq 1$. Moreover, the matrix representation of $L_{M_z}$ with respect to the orthonormal basis $\{f_n\}_{n \geq 0}$ is given by
\[
[L_{M_z}] = \begin{bmatrix}
0 & \frac{a_1}{a_0} & 0 & 0  & 0 & \dots
\\
0 & {d_1} & \frac{a_2}{a_1} & 0  & 0 & \ddots
\\
0 & \frac{-d_1b_1}{a_2} & d_2 & \frac{a_3}{a_2} & 0 & \ddots
\\
0 & \frac{d_1b_1b_2}{a_2a_3} &\frac{-d_2b_2}{a_3} & d_3 &\frac{a_4}{a_3} & \ddots
\\
0 & \frac{-d_1b_1b_2b_3}{a_2a_3a_4} &\frac{d_2b_2b_3}{a_3a_4} & \frac{-d_3b_3}{a_4} & d_4 & \ddots
\\
\vdots & \vdots & \vdots&\vdots &\ddots &\ddots
\end{bmatrix}.
\]
\end{Theorem}

Next we verify that the bounded away assumption of $\{|\frac{a_n}{a_{n+1}}|\}_{n\geq 0}$ in \eqref{eqn: bounded away an/bn} is also a necessary condition for left-invertible shifts.

\begin{Theorem}\label{thm: iff bounded away}
Let $\clh_k$ be a semi-analytic tridiagonal space corresponding to the orthonormal basis $\{f_n\}_{n\geq 0}$, where $f_n(z) = (a_n + b_n z) z^n$, $n\geq 0$. Then $M_z$ is left-invertible if and only if $\{|\frac{a_n}{a_{n+1}}|\}_{n\geq 0}$ is bounded away from zero, or equivalently, $\clh_k$ is an analytic tridiagonal space.
\end{Theorem}
\begin{proof}
In view of Theorem \ref{thm: L M_z = I} we only need to prove the necessary part. Consider the Shimorin left inverse $L_{M_z} = (M_z^* M_z)^{-1} M_z^*$. Using the fact that $\mathbb{C}[z] \subseteq \clh_k$, one can show, along the similar line of computation preceding Theorem \ref{Prop: L Matrix} (note that, by assumption, $L_{M_z}$ is bounded), that the matrix representation of $L_{M_z}$ with respect to the orthonormal basis $\{f_n\}_{n \geq 0}$ is precisely given by the one in Theorem \ref{Prop: L Matrix}. Then for each $n \geq 0$, we have
\[
\begin{split}
\|(M_z^* M_z)^{-1} M_z^*\|_{\clb(\clh_k)} \geq \|(M_z^* M_z)^{-1} M_z^* f_n\|_{\clh_k} \geq \Big|\frac{a_{n+1}}{a_{n}}\Big|,
\end{split}
\]
which implies that
\[
\Big|\frac{a_n}{a_{n+1}}\Big| \geq \frac{1}{\|(M_z^* M_z)^{-1} M_z^*\|_{\clb(\clh_k)}},
\]
and hence the sequence is bounded away from zero.
\end{proof}

\newsection{Tridiagonal Shimorin models}\label{sect: S kernels and TDS}

As emphasized already in Proposition \ref{Prop:shimorin diagonal} that if $k$ is a diagonal kernel, then ${k_{M_z}}$ is also a diagonal kernel. However, as we will see in the example below, Shimorin kernels are not compatible with tridiagonal kernels. This consequently motivates one to ask: How to determine whether or not the Shimorin kernel $k_{M_z}$ of a tridiagonal kernel $k$ is also tridiagonal? We have a complete answer to this question: $k_{M_z}$ is tridiagonal if and only if $b_0=0$ or that $M_z$ is a weighted shift on $\clh_k$. This is the main content of this section.

\begin{Example}\label{example: TD RKHS fail}
Let $a_n = 1$ for all $n \geq 0$, $b_0 = \frac{1}{2}$, and let $b_n = 0$ for all $n \geq 1$. Let $\clh_k$ denote the analytic tridiagonal space corresponding to the orthonormal basis $\{f_n\}_{n \geq 0}$, where $f_n = (a_n + b_n z) z^n$ for all $n \geq 0$. Since $f_0 = 1 + \frac{1}{2} z$ and $f_n = z^n$ for all $n\geq 1$, by \eqref{eqn: Mz matrix}, we have
\[
[M_z] = \begin{bmatrix}
0& 0 & 0 & 0 & \dots
\\
1 & 0 & 0 & 0 & \ddots
\\
\frac{1}{2} & 1 & 0 & 0 & \ddots
\\
0 & 0 & 1 & 0 &  \ddots
\\
\vdots & \vdots & \vdots& \ddots & \ddots
\end{bmatrix}.
\]
By Theorem \ref{Prop: L Matrix}, the Shimorin left inverse $L_{M_z} = (M_z^*M_z)^{-1}M_z^*$ is given by
\[
L_{M_z} = \begin{bmatrix}
0 & 1 & 0  & 0  & 0 & \dots
\\
0 & \frac{-1}{2} & 1 & 0 & 0 & \ddots
\\
0 & 0 & 0  & 1 & 0 & \ddots
\\
0 & 0 & 0 & 0 & 1 & \ddots
\\
\vdots & \vdots & \vdots & \vdots & \ddots & \ddots
\end{bmatrix}.
\]
Recall, in this case, that $\clw = \mathbb{C} f_0$. It is easy to
check that $L_{M_z} f_1 = f_0 - \frac{1}{2} f_1$, $L_{M_z}^* f_0 = f_1$, $L_{M_z}^* f_1 = - \frac{1}{2} f_1 + f_2$, and $L_{M_z}^* f_2 = f_3$. Then
\[
L_{M_z}^{*3} f_0 = - \frac{1}{2} L_{M_z}^* f_1 + L_{M_z}^* f_2 = \frac{1}{4} f_1 - \frac{1}{2} f_2 + f_3,
\]
and hence $P_{\clw} L_{M_z} L_{M_z}^{*3} f_0 = \frac{1}{4} P_{\clw} (L_{M_z} f_1)$, as $P_{\clw} L_{M_z} f_j = 0$ for all $j \neq 1$. Consequently
\[
P_{\clw} L_{M_z} L_{M_z}^{*3} f_0 = \frac{1}{4} f_0 \neq 0,
\]
which implies that the Shimorin kernel ${k}_{M_z}$, as defined in \eqref{eqn: k_T(z,w)}, is not a tridiagonal kernel.
\end{Example}

Throughout this section, $\clh_k$ will be an analytic tridiagonal space corresponding to the orthonormal basis $\{f_n\}_{n\geq 0}$, where $f_n(z) = (a_n + b_nz) z^n$, $n \geq 0$. Recall that the Shimorin kernel $k_{M_z}: \D \times \D \raro \clb(\clw)$ is given by (see \eqref{eqn: k_T(z,w)} and also Theorem \ref{thm-Shimorin}) 
\[
k_{M_z}(z, w) = P_{\clw} (I - z L_{M_z})^{-1} (I - \bar{w} L_{M_z}^*)^{-1}|_{\clw} \quad \quad (z, w \in \D).
\]
Here, of course, $\clw = \mathbb{C} f_0$, the one-dimensional space generated by the vector $f_0$. So one may regard $k_{M_z}$ as a scalar kernel. We are now ready for the main result of this section.

\begin{Theorem}\label{thm: shimorin classify}
The Shimorin kernel $k_{M_z}$ of $M_z$ is tridiagonal if and only if $M_z$ on $\clh_k$ is a weighted shift or
\[
b_0=0.
\]
\end{Theorem}
\begin{proof}
We split the proof into several steps.

\NI\textsf{Step 1:} We first denote $L_{M_z} = L$ and
\[
X_{m n} = P_{\clw} L^m L^{*n}|_{\clw} \quad \quad (m, n \geq 0),
\]
for simplicity. First observe that Theorem \ref{Prop: L Matrix} implies that $L^m f_0 = 0$, $m \geq 1$, and hence, $X_{m0} = 0 = X_{m0}^* = X_{0m}$ for all $m \geq 1$. Then the formal matrix representation of the Shimorin kernel $k_{M_z}$ is given by
\begin{equation}\label{eqn:[k Mz]}
[k_{M_z}] =
\begin{bmatrix}
I_{\clw} & 0 & 0  & 0 & \dots
\\
0 & X_{1 1} & X_{1 2} & X_{1 3} &  \dots
\\
0 & X_{12}^* & X_{2 2} & X_{2 3} &  \dots
\\
0 & X_{13}^* & X_{2 3}^* & X_{3 3} &  \dots
\\
\vdots & \vdots & \vdots & \ddots &\ddots
\end{bmatrix}.
\end{equation}
Clearly, in view of the above, $k_{M_z}$ is tridiagonal if and only if $X_{m n} f_0 = 0$ for all $m, n \neq 0$ and $|m-n| \geq 2$.

\NI\textsf{Step 2:} In this step we aim to compute matrix representations of $L^p$ and $L^{*p}$, $p \geq 1$, with respect to the orthonormal basis $\{f_n\}_{n\geq 0}$. The matrix representation of $[L]$ in Theorem \ref{Prop: L Matrix} is instructive. It also follows that
\begin{equation}\label{eqn:[L*]}
[L^*] =
\begin{bmatrix}
0 & 0 & 0  & 0 & 0 & \dots
\\
\frac{\bar{a}_1}{\bar{a}_0} & \bar{d}_1 & \frac{-\bar{d}_1 \bar{b}_1}{\bar{a}_2} & \frac{\bar{d}_1 \bar{b}_1 \bar{b}_2}{\bar{a}_2 \bar{a}_3} & \frac{- \bar{d}_1 \bar{b}_1 \bar{b}_2 \bar{b}_3}{\bar{a}_2 \bar{a}_3 \bar{a}_4} &  \dots
\\
0 & \frac{\bar{a}_2}{\bar{a}_1} & \bar{d}_2 & \frac{-\bar{d}_2 \bar{b}_2}{\bar{a}_3} & \frac{\bar{d}_2 \bar{b}_2 \bar{b}_3}{\bar{a}_3 \bar{a}_4} &  \ddots
\\
0 & 0 & \frac{\bar{a}_3}{\bar{a}_2} & \bar{d}_3 & \frac{-\bar{d}_3 \bar{b}_3}{\bar{a}_4} &  \ddots
\\
0 & 0 & 0 & \frac{\bar{a}_4}{\bar{a}_3} & \bar{d}_4 &  \ddots
\\
0 & 0 & 0 & 0 & \frac{\bar{a}_5}{\bar{a}_4} & \ddots
\\
\vdots & \vdots & \vdots & \vdots &\ddots & \ddots
\end{bmatrix}.
\end{equation}
Here we redo the construction taking into account the general $p \geq 1$, and proceed as in the proof of Theorem \ref{Prop: L Matrix}. However, the proof is by no means the same and the general case is quite involved. Assume that $n \geq 1$. We need to consider two cases: $n \geq p$ and $n \leq p-1$. Suppose $n \geq p$. By \eqref{eqn:L_Mz back shift} and \eqref{eqn:L_1}, we have
\[
L^p f_n  = a_n L^p z^n + b_n L^p z^{n+1} = a_n z^{n-p} + b_ n z^{n-p+1},
\]
which implies
\[
L^p f_n  = \frac{a_n}{a_{n-p}}(a_{n-p}z^{n-p} + b_{n-p}z^{n-p+1})+(b_n-\frac{a_n}{a_{n-p}}b_{n-p})z^{n-p+1} = \frac{a_n}{a_{n-p}} f_{n-p} + d_n^{(p)} z^{n-p+1},
\]
where
\begin{equation}\label{eqn: d^p_n}
d_n^{(p)}= b_n-\frac{a_n}{a_{n-p}}b_{n-p} \quad\quad (n \geq p).
\end{equation}
Hence by \eqref{eq: z^n formula}
\[
L^p f_n  = \frac{a_n}{a_{n-p}} f_{n-p} + \frac{d_n^{(p)}}{a_{n-p+1}} \Big(f_{n-p+1} - \frac{b_{n-p+1}}{a_{n-p+2}} f_{n-p+2} + \frac{b_{n-p+1} b_{n-p+2}}{a_{n-p+2}a_{n-p+3}} f_{n-p+3} - \cdots \Big),
\]
that is
\[
L^p f_n  = \frac{a_n}{a_{n-p}} f_{n-p} + \frac{d_n^{(p)}}{a_{n-p+1}} \sum_{m=0}^{\infty} (-1)^m \Big(\frac{\prod_{j=0}^{m-1} b_{n - p+j+1}}{\prod_{j=0}^{m-1} a_{n - p+j+2}}\Big) f_{n-p+m+1},
\]
for all $n \geq p$. Here and in what follows, we define $\prod_{j=0}^{-1} x_j := 1$.

\NI We now let $p=1$ and $n=1$. Then by Theorem \ref{Prop: L Matrix}, we have
\begin{equation}\label{eqn:L f1}
L f_1 = \frac{a_1}{a_0} f_0 + d_1 f_1 + (- \frac{d_1 b_1}{a_2}) f_2 + (\frac{d_1 b_1 b_2}{a_2 a_3}) f_3 + \cdots.
\end{equation}

\NI Finally, let $1 \leq n \leq p-1$. Then $p > 1$, and again by \eqref{eqn:L_Mz back shift} and \eqref{eqn:L_1}, we have
\[
L^p f_n = L^p(a_n z^n+b_n z^{n+1}) = a_n L^{p-n} 1 + b_n L^{p-n-1} 1 = a_n \Big(\frac{-b_0}{a_0}\Big)^{p-n} + b_n \Big(\frac{-b_0}{a_0}\Big)^{p-n-1},
\]
and hence $L^p f_n = a_n \Big(\frac{-b_0}{a_0}\Big)^{p-n-1}\Big[\frac{b_n}{a_n}-\frac{b_0}{a_0}\Big]$. We set
\begin{equation}\label{eqn: beta n}
\beta_n = \frac{b_n}{a_n}-\frac{b_0}{a_0} \quad \quad (n \geq 1),
\end{equation}
and
\begin{equation}\label{eqn: beta n p}
\beta_n^{(p)} =
a_n \Big(\frac{-b_0}{a_0}\Big)^{p-n-1} \beta_n  \quad \quad (1 \leq n \leq p-1).
\end{equation}
Then $L^p f_n = \beta_n^{(p)}$ and \eqref{eq: z^n formula} implies that
\[
L^p(f_n) = \frac{\beta_n^{(p)}}{a_0} \sum_{m=0}^\infty (-1)^m \Big(\frac{\Pi_{j=0}^{m-1} b_j}{\Pi_{j=0}^{m-1} a_{j+1}}\Big) f_m,
\]
for all $1 \leq n \leq p-1$. Then
\begin{equation}\label{eqn:[L2]}
[L^2] =
\begin{bmatrix}
0 & \frac{\beta^{(2)}_1}{a_0} & \frac{a_2}{a_0} & 0 & 0 & \dots
\\
0 & - \frac{\beta^{(2)}_1 b_0}{a_0 a_1} & \frac{d_2^{(2)}}{a_1} & \frac{a_3}{a_1} & 0 & \ddots
\\
0 & \frac{\beta^{(2)}_1 b_0 b_1}{a_0 a_1 a_2} & - \frac{d_2^{(2)}b_1}{a_1 a_2} & \frac{d_3^{(2)}}{a_2} & \frac{a_4}{a_2} & \ddots
\\
0 & - \frac{\beta^{(2)}_1 b_0 b_1 b_2}{a_0 a_1 a_2 a_3} & \frac{d_2^{(2)}b_1 b_2}{a_1 a_2 a_3} & - \frac{d_3^{(2)} b_2}{a_2 a_3} & \frac{d_4^{(2)}}{a_3} & \ddots
\\
\vdots & \vdots & \vdots & \vdots & \ddots & \ddots
\end{bmatrix},
\end{equation}
and in general, for each $p \geq 2$, we have
\begin{equation}\label{eqn:Lp}
[L^p] = \begin{bmatrix}
0 & \frac{ \beta_1^{(p)}}{a_0} & \frac{\beta_2^{(p)}}{a_0} & \cdots  &  \frac{\beta_{p-1}^{(p)}}{a_0} & \frac{a_p}{a_0} & 0 & 0 & \cdots
\\
0 & - \frac{ \beta_1^{(p)}b_0}{a_0a_1} & - \frac{ \beta_2^{(p)}b_0}{a_0a_1} & \cdots  &  -\frac{\beta_{p-1}^{(p)}b_0}{a_0a_1} & \frac{d_p^{(p)}}{a_1}   & \frac{a_{p+1}}{a_1} & 0 & \ddots
\\
0 & \frac{ \beta_1^{(p)}b_0b_1}{a_0a_1a_2} & \frac{ \beta_2^{(p)}b_0b_1}{a_0a_1a_2} & \cdots  &  \frac{\beta_{p-1}^{(p)}b_0b_1}{a_0a_1a_2} & -\frac{d_{p}^{(p)}b_1}{a_1a_2} & \frac{d_{p+1}^{(p)}}{a_2} & \frac{a_{p+2}}{a_2} & \ddots
\\
0 & -\frac{ \beta_1^{(p)}b_0 b_1 b_2}{a_0a_1a_2 a_3} & - \frac{\beta_2^{(p)}b_0 b_1 b_2}{a_0 a_1 a_2 a_3} & \cdots  & -\frac{\beta_{p-1}^{(p)}b_0b_1 b_2}{a_0a_1a_2 a_3} & \frac{d_{p}^{(p)}b_1 b_2}{a_1a_2 a_3} & - \frac{d_{p+1}^{(p)}b_2}{a_2 a_3} & \frac{d^{(p)}_{p+2}}{a_3} & \ddots
\\
\vdots & \vdots & \vdots &\vdots &\vdots &\vdots &\vdots &\ddots &\ddots
\end{bmatrix}.
\end{equation}
Hence, for each $p \geq 2$, we have
\begin{equation}\label{eqn:L*p}
[L^{*p}] = \begin{bmatrix}
0 & 0 & 0 & 0 & \dots
\\
\frac{\bar{\beta}_1^{(p)}}{\bar{a}_0} & - \frac{\bar{\beta}_1^{(p)} \bar{b}_0}{\bar{a}_0 \bar{a}_1}  & \frac{\bar{\beta}_1^{(p)} \bar{b}_0 \bar{b}_1}{\bar{a}_0 \bar{a}_1 \bar{a}_2} & -\frac{\bar{\beta}_1^{(p)} \bar{b}_0 \bar{b}_1 \bar{b}_2}{\bar{a}_0 \bar{a}_1 \bar{a}_2 \bar{a}_3} & \ddots
\\
\frac{\bar{\beta}_2^{(p)}}{\bar{a}_0} & - \frac{\bar{\beta}_2^{(p)} \bar{b}_0}{\bar{a}_0 \bar{a}_1}  & \frac{\bar{\beta}_2^{(p)} \bar{b}_0 \bar{b}_1}{\bar{a}_0 \bar{a}_1 \bar{a}_2} & -\frac{\bar{\beta}_2^{(p)} \bar{b}_0 \bar{b}_1 \bar{b}_2}{\bar{a}_0 \bar{a}_1 \bar{a}_2 \bar{a}_3}  & \ddots
\\
\vdots & \vdots & \vdots & \vdots  & \ddots
\\
\frac{\bar{\beta}_{p-1}^{(p)}}{\bar{a}_0} & - \frac{\bar{\beta}_{p-1}^{(p)} \bar{b}_0}{\bar{a}_0 \bar{a}_1}  & \frac{\bar{\beta}_{p-1}^{(p)} \bar{b}_0 \bar{b}_1}{\bar{a}_0 \bar{a}_1 \bar{a}_2} & -\frac{ \bar{\beta}_{p-1}^{(p)} \bar{b}_0 \bar{b}_1 \bar{b}_2}{\bar{a}_0 \bar{a}_1 \bar{a}_2 \bar{a}_3} & \ddots
\\
\frac{\bar{a}_p}{\bar{a}_0} &  \frac{\bar{d}_p^{(p)}}{\bar{a}_1} & - \frac{\bar{d}_{p}^{(p)} \bar{b}_1}{\bar{a}_1\bar{a}_2} & \frac{\bar{d}_{p}^{(p)} \bar{b}_1 \bar{b}_2}{\bar{a}_1 \bar{a}_2 \bar{a}_3} & \ddots
\\
0 & \frac{\bar{a}_{p+1}}{\bar{a}_1} & \frac{\bar{d}_{p+1}^{(p)}}{\bar{a}_2} & -\frac{\bar{d}_{p+1}^{(p)} \bar{b}_2}{\bar{a}_2 \bar{a}_3} & \ddots
\\
0 & 0 & \frac{\bar{a}_{p+2}}{\bar{a}_2} & \frac{\bar{d}_{p+2}^{(p)}}{\bar{a}_3} & \ddots
\\
\vdots & \vdots & \vdots & \vdots & \ddots
\end{bmatrix}.
\end{equation}

\NI\textsf{Step 3:} We now identify condition on the sequence $\{\beta^{(n+2)}_{n}\}_{n\geq 1}$ implied by the requirement that $X_{m, m+2} = 0$, $m \geq 1$. Before proceeding further, we record here the following crucial observation: Suppose $\beta^{(p)}_n=0$ for some $p$ and $n$ such that $1 \leq n \leq p-1$. Then by \eqref{eqn: beta n p}, we have
\begin{equation}\label{eqn:beta p n = 0}
\beta^{(q)}_n=0 \quad \quad (q \geq p).
\end{equation}
Now assume $m \geq 1$. The matrix representation in \eqref{eqn:L*p} implies
\begin{equation}\label{eqn:L*(n+2)f0}
L^{* m+2} f_0 = \frac{1}{\bar{a}_0}\Big(\bar{\beta}_1^{(m+2)} f_1 + \bar{\beta}_2^{(m+2)} f_2 + \cdots + \bar{\beta}_{m+1}^{(m+2)} f_{m+1} +  \bar{a}_{m+2} f_{m+2}\Big).
\end{equation}
Observe that, by Theorem \ref{Prop: L Matrix}, we have
\[
P_{\clw} L(f_i) =
\begin{cases}
\frac{a_1}{a_0} f_0 & \mbox{if}~ i=1
\\
0 & \mbox{if}~ i \neq 1.
\end{cases}
\]
Let us now assume that $m \geq 2$. Then \eqref{eqn:Lp} implies
\begin{equation}\label{eqn:L*n fj}
P_{\clw} L^m(f_i) =
\begin{cases}
\frac{\beta^{(m)}_{i}}{a_0} f_0 & \mbox{if}~ 1 \leq i \leq m-1
\\
\frac{a_m}{a_0} f_0 & \mbox{if}~ i=m
\\
0 & \mbox{if}~ i \geq m+1. \end{cases}
\end{equation}
Since $X_{m, m+2} = P_{\clw}L^m L^{* m+2}|_{\clw}$, this yields
\begin{equation}\label{eqn:LnL*n+2f0}
X_{m, m+2}f_0 = \frac{1}{|{a}_0|^2} \Big(\bar{\beta}_1^{(m+2)} \beta_1^{(m)} + \bar{\beta}_2^{(m+2)} \beta_2^{(m)} + \cdots + \bar{\beta}_{m-1}^{(m+2)} \beta_{m-1}^{(m)} + \bar{\beta}_{m}^{(m+2)} a_m \Big) f_0.
\end{equation}
In particular, if $m=1$, then we have
\[
X_{13} f_0 = \frac{1}{\bar{a}_0} \Big(\bar{\beta}^{(3)}_1 \frac{a_1}{a_0}\Big) f_0,
\]
and hence $X_{13} = 0$ if and only if $\beta^{(3)}_1 = 0$. By \eqref{eqn:LnL*n+2f0}, applied with $m=2$ we have
\[
X_{24} f_0 = \frac{1}{|a_0|^2} \Big(\bar{\beta}_{1}^{(4)} \beta_{1}^{(2)} + \bar{\beta}_{2}^{(4)} a_2\Big) f_0.
\]
Assume that $\beta^{(3)}_1 = 0$. By \eqref{eqn:beta p n = 0}, we have $\beta_{1}^{(4)} = 0$, and, consequently
\[
X_{24} f_0 = \bar{\beta}_{2}^{(4)} \frac{a_2}{|a_0|^2} f_0.
\]
Hence we obtain $X_{24} = 0$ if and only if $\beta^{(4)}_2 = 0$. Therefore, if $X_{m,m+2} = 0$ for all $m \geq 1$, then by induction, it follows that $\beta^{(m+2)}_{m} = 0$ for all $m \geq 1$. The converse also follows from the above computation.

Thus we have proved: \textit{$X_{m,m+2} = 0$ for all $m \geq 1$ if and only if $\beta^{(m+2)}_{m} = 0$ for all $m \geq 1$.}

\smallskip

\NI\textsf{Step 4:} Our aim is to prove the following claim: \textit{Suppose $X_{i, i+2} = 0$ for all $i=1, \ldots, m$, and $m \geq 1$. Then $X_{m n} = 0$ for all $n = m+3, m+4, \ldots$, and $m \geq 1$.}

\noindent To this end, let $n = m+j$ and $j \geq 3$. Then the matrix representation in \eqref{eqn:L*p} (or the equality \eqref{eqn:L*(n+2)f0}) implies
\[
L^{*n} f_0 = \frac{1}{\bar{a}_0}\Big(\bar{\beta}_1^{(n)} f_1 + \bar{\beta}_2^{(n)} f_2 + \cdots + \bar{\beta}_{n-1}^{(n)} f_{n-1} +  \bar{a}_{n} f_{n}\Big),
\]
and then
\[
\begin{split}
P_{\clw}L^m L^{*n} f_0 = \Big(\frac{1}{\bar{a}_0} \sum_{i=1}^{n-1} \bar{\beta}_i^{(n)} P_{\clw} L^m (f_i)\Big) + \frac{\bar{a}_{n}}{\bar{a}_0} P_{\clw} L^m f_n = \frac{1}{\bar{a}_0} \sum_{i=1}^{m} \bar{\beta}_i^{(n)} P_{\clw} L^m (f_i),
\end{split}
\]
since $P_{\clw} L^m f_i = 0$, $i > m$, which follows from the matrix representation of $L^m$ in \eqref{eqn:Lp}. Hence by \eqref{eqn:L*n fj} (or directly from \eqref{eqn:Lp}), we have
\[
P_{\clw} L^m L^{*n} f_0 = \frac{1}{|a_0|^2} \Big(\bar{\beta}_1^{(n)} \beta_1^{(m)} + \bar{\beta}_2^{(n)} \beta_2^{(m)} + \cdots + \bar{\beta}_{m-1}^{(n)} \beta_{m-1}^{(m)} +  {a}_{m} \bar{\beta}_{m}^{(n)}\Big) f_0.
\]
Now note that $X_{i, i+2} = 0$, that is, $\beta_i^{(i+2)} = 0$, $i= 1, \ldots,m$, by assumption. Since $i+2 \leq m+j$ for all $i=1, \ldots, m$, by \eqref{eqn:beta p n = 0}, we have
\[
\beta_i^{(n)} = \beta_i^{(m+j)} = 0 \quad \quad (i =1, \ldots, m).
\]
Hence $P_{\clw} L^m L^{*n} f_0 = 0$, that is, $X_{m, m+i} = 0$, $i= 3, 4, \ldots$, which proves the claim.

\NI\textsf{Step 5:} So far all we have proved is that $X_{mn} = 0$ for all $|m-n| \geq 2$ if and only if $\beta_m^{(m+2)} = 0$ for all $m \geq 1$. Now, by \eqref{eqn: beta n p} and \eqref{eqn: beta n}, we have
\[
\beta_n^{(n+2)} = a_n \Big(- \frac{b_0}{a_0}\Big) \beta_n,
\]
where $\beta_n = \frac{b_n}{a_n} - \frac{b_0}{a_0}$ for all $n \geq 1$. Thus $\beta_n^{(n+2)} = 0$ for all $n \geq 1$ if and only if $b_0 = 0$ or $\beta_n = 0$ for all $n \geq 1$. On the other hand, Lemma \ref{lemma: M_z weighted shift} implies that $\beta_n = 0$ for all $n \geq 1$ if and only if $M_z$ is a weighted shift.

\NI Finally, by Proposition \ref{Prop:shimorin diagonal}, we know that if $M_z$ is a left-invertible weighted shift, then the Shimorin kernel is also a diagonal kernel. This completes the proof of Theorem \ref{thm: shimorin classify}.
\end{proof}

\newsection{Positive operators and tridiagonal kernels}\label{sect: Positive and TDK}

Our aim is to classify positive operators $P$ on a tridiagonal space $\clh_k$ such that
\[
\D \times \D \ni(z, w) \mapsto \BL  P k(\cdot, w) , k(\cdot, z) \BR_{\clh_k},
\]
is also a tridiagonal kernel. While this problem is of independent interest, the motivation for our interest in this question also comes from Theorem \ref{thm: revisit model} (also see the paragraph preceding Corollary \ref{cor: M TD implies M TD}). We start with a simple example.

\begin{Example}
We consider the same example as in Example \ref{example: TD RKHS fail}. Note that $M_z$ is left-invertible and not a weighted shift with respect to the orthonormal basis $\{f_n\}_{n\geq 0}$ of $\clh_k$. Then by Lemma \ref{prop: L tilde T}, we have
\[
|M_z|^{-2} = L_{M_z} L_{M_z}^* =
\begin{bmatrix}
1 & -\frac{1}{2} & 0  & 0 & \dots
\\
-\frac{1}{2} & \frac{5}{4} & 0 & 0 & \ddots
\\
0 & 0 & 1 & 0 & \ddots
\\
0 & 0 & 0 & 1 & \ddots
\\
\vdots & \vdots & \vdots &\vdots &\ddots
\end{bmatrix}.
\]
Let
\[
|M_z|^{-1}= \begin{bmatrix}
\alpha & \beta & 0  & 0  & \dots
\\
\beta & \gamma & 0 & 0 & \ddots
\\
0 & 0 & 1  & 0 & \ddots
\\
0 & 0 & 0 & 1 &  \ddots
\\
\vdots & \vdots & \vdots &\vdots  &\ddots
\end{bmatrix},
\]
where $\begin{bmatrix}\alpha & \beta \\ \beta & \gamma \end{bmatrix}$ is the positive square root of $\begin{bmatrix} 1 & -\frac{1}{2}\\ -\frac{1}{2} & \frac{5}{4}\end{bmatrix}$. A straightforward calculation shows that $\frac{\alpha}{2} + \beta \neq 0$. Define $K : \D \times \D \raro \mathbb{C}$ by
\[
K(z,w) = \BL |M_z|^{-1} k(\cdot, w) , k(\cdot, z) \BR_{\clh_k} \quad \quad (z, w \in \D).
\]
A simple computation then shows that
\[
K(z,w)=\alpha+(\frac{\alpha}{2} + \beta)\bar{w}+(\frac{\alpha}{2} + \beta) z + (\frac{\alpha}{4} + \beta + \gamma) z \bar{w} + \sum_{n\geq 2} z^n \bar{w}^n ,
\]
that is, $K$ is also a tridiagonal kernel.
\end{Example}

The following is a complete classification of positive operators $P$ for which $(z, w) \mapsto \BL  P k(\cdot, w) , k(\cdot, z) \BR_{\clh_k}$ defines a  tridiagonal kernel.

\begin{Theorem}\label{thm: P kernel and Mz}
Let $\clh_k$ be a tridiagonal space corresponding to the orthonormal basis $f_n(z) = (a_n + b_nz)z^n$, $n \geq 0$. Let $P$ be a positive operator on $\clh_k$ with matrix representation
\[
P = \begin{bmatrix}
c_{00}& c_{01} & c_{02} & c_{03} & \dots
\\
\bar{c}_{01} & c_{11} & c_{12} & c_{13} & \ddots
\\
\bar{c}_{02} & \bar{c}_{12} & c_{22} & c_{23}  & \ddots
\\
\bar{c}_{03} & \bar{c}_{13} & \bar{c}_{23} & c_{33} &
\ddots
\\
\vdots & \vdots & \vdots & \ddots & \ddots
\\
\end{bmatrix},
\]
with respect to the basis $\{f_n\}_{n \geq 0}$. Then the positive definite scalar kernel $K$, defined by
\[
K(z,w) = \BL P k(\cdot, w) ,k(\cdot, z) \BR_{\clh_k} \quad\quad (z, w \in \D),
\]
is tridiagonal if and only if
\[
c_{0n} = (-1)^{n-1} \frac{\bar{b}_1 \cdots \bar{b}_{n-1}}{\bar{a}_2 \cdots \bar{a}_n} c_{01} \quad \quad (n \geq 2),
\]
and
\[
c_{mn} = (-1)^{n-m-1} \frac{\bar{b}_{m+1} \cdots \bar{b}_{n-1}}{\bar{a}_{m+2} \cdots \bar{a}_n} c_{m, m+1} \qquad (1 \leq m \leq n-2).
\]
Equivalently, $K$ is tridiagonal if and only if
\[
P = \begin{bmatrix}
c_{00} & c_{01} & -\frac{\bar{b_1}}{\bar{a_2}}c_{01} & \frac{\bar{b}_1 \bar{b}_2}{\bar{a}_2 \bar{a}_3} c_{01} & \dots
\\
\bar{c}_{01} & c_{11} & c_{12}  &-\frac{\bar{b}_2} {\bar{a}_3}c_{12} & \ddots
\\
-\frac{b_1}{a_2}\bar{c}_{01} & \bar{c}_{12} & c_{22} &c_{23} & \ddots
\\
\frac{b_1b_2}{a_2a_3}\bar{c}_{01} & -\frac{b_2}{a_3}\bar{c}_{12} & \bar{c}_{23} & c_{33} & \ddots
\\
\vdots & \vdots & \vdots & \ddots & \ddots
\end{bmatrix}.
\]
\end{Theorem}

\begin{proof}
Note, for each $w \in \D$, by \eqref{eqn intro: k(z,w)}, we have $k(\cdot, w) = \sum_{m=0}^\infty \overline{f_m(w)} f_m$, and thus
\[
P k(\cdot, w) = \sum_{m=0}^{\infty} (\sum_{n=0}^{m-1} \bar{c}_{nm} \overline{f_n(w)} + \sum_{n=m}^{\infty} c_{mn} \overline{f_n(w)}) f_m,
\]
where $\sum_{n=0}^{-1} x_n :=0$. Then
\[
\begin{split}
\BL P k(\cdot, w) ,k(\cdot, z) \BR_{\clh_k} & = \sum_{m=0}^{\infty} f_m(z) (\sum_{n=0}^{m-1} \bar{c}_{nm} \overline{f_n(w)} + \sum_{n=m}^{\infty} c_{mn} \overline{f_n(w)})
\\
& = \sum_{m=0}^{\infty} (a_m z^m + b_m z^{m+1}) (\sum_{n=0}^{m-1} \bar{c}_{nm} (\bar{a}_n \bar{w}^n + \bar{b}_n \bar{w}^{n+1})
\\
& \quad \quad + \sum_{n=m}^{\infty} c_{mn} (\bar{a}_n \bar{w}^n + \bar{b}_n \bar{w}^{n+1}))
\\
& = \sum_{m,n \geq 0} \alpha_{mn} z^m \bar{w}^n,
\end{split}
\]
where $\alpha_{mn}$ denotes the coefficient of $z^m \bar{w}^n$, $m, n \geq 0$. Our interest here is to compute $\alpha_{mn}$, $|m-n| \geq 2$. Clearly, $\alpha_{mn} = \bar{\alpha}_{nm}$ for all $m, n \geq 0$, and
\begin{equation}\label{eqn:alpha 0n}
\alpha_{0n} = a_0 (\bar{a}_n c_{0 n} + \bar{b}_{n-1} c_{0,n-1}) \quad \quad (n \geq 2),
\end{equation}
and
\begin{equation}\label{eqn:alpha mn}
\alpha_{mn} = a_m\Big(\bar{a}_n c_{mn} + \bar{b}_{n-1} c_{m, n-1} \Big) + b_{m-1} \Big( \bar{a}_n c_{m-1, n} + \bar{b}_{n-1} c_{m-1, n-1} \Big)
\quad \quad (1 \leq m<n).
\end{equation}
Suppose $n \geq 2$. By \eqref{eqn:alpha 0n}, $\alpha_{0n}=0$ if and only if $c_{0n} = - \frac{\bar{b}_{n-1}}{\bar{a}_n} c_{0, n-1}$. In particular, if $n=2$, then $c_{02} = -\frac{\bar{b}_1}{\bar{a}_2}c_{01}$, and hence, by \eqref{eqn:alpha 0n} again, we have
\[
c_{0n} = (-1)^{n-1} \frac{\prod_{i=1}^{n-1} \bar{b}_i}{\prod_{i=2}^n \bar{a}_i} c_{01} \quad \quad (n \geq 2).
\]
Therefore, $\alpha_{0n}=0$ for all $n \geq 2$ if and only if the above identity hold for all $n \geq 2$.

\NI Next we want to consider the case $m, n \neq 0$ and $|m-n| \geq 2$. Assume that $n \geq 3$. Then \eqref{eqn:alpha mn} along with \eqref{eqn:alpha 0n} implies
\[
\alpha_{1n} = a_1(\bar{a}_n c_{1n} + \bar{b}_{n-1} c_{1, n-1}) + b_{0}(\bar{a}_n c_{0n} + \bar{b}_{n-1} c_{0, n-1}) = a_1(\bar{a}_n c_{1n} + \bar{b}_{n-1} c_{1, n-1}) + \frac{b_{0}}{a_0} \alpha_{0n}.
\]
Therefore, if $\alpha_{0n}=0$ for all $n \geq 3$, then $\alpha_{1n} = a_1(\bar{a}_n c_{1n} + \bar{b}_{n-1} c_{1, n-1})$. Hence $\alpha_{1n} = 0$ if and only if $\bar{a}_n c_{1n} + \bar{b}_{n-1} c_{1, n-1} = 0$, which is equivalent to
\[
c_{1n} = - \frac{\bar{b}_{n-1}}{\bar{a}_n} c_{1, n-1}.
\]
Therefore, under the assumption that $\alpha_{1n} = 0$ and $n \geq 4$, \eqref{eqn:alpha mn} along with \eqref{eqn:alpha 0n} implies
\[
\alpha_{2n} = a_2(\bar{a}_n c_{2n} + \bar{b}_{n-1} c_{2, n-1}) + b_1 (\bar{a}_n c_{1n} + \bar{b}_{n-1} c_{1, n-1}) = a_2 (\bar{a}_n c_{2n} + \bar{b}_{n-1} c_{2, n-1}).
\]
Then $\alpha_{2n} = 0$, $n \geq 4$, if and only if $c_{2n} = - \frac{\bar{b}_{n-1}}{\bar{a}_n} c_{2, n-1}$. Consequently, by induction, for all $m, n \neq 0$ and $|m-n| \geq 2$, we have that $\alpha_{mn} = 0$ if and only if $\bar{a}_n c_{mn} + \bar{b}_{n-1} c_{m, n-1} = 0$, or equivalently
\[
c_{mn} = - \frac{\bar{b}_{n-1}}{\bar{a}_n} c_{m, n-1}.
\]
Finally, observe that $c_{mn} = (-1)^{n-m-1} \frac{\bar{b}_{n-1} \cdots \bar{b}_{m+1}}{\bar{a}_n \cdots \bar{a}_{m+2}} c_{m, m+1}$ for all $1 \leq m \leq n-2$. This completes the proof of the theorem.
\end{proof}

We will return to this in  Theorem \ref{thm: truncated} and Corollary \ref{cor:truncated SK = SK}.

\newsection{Quasinormal operators}\label{sect: quasi normal}

A bounded linear operator $T \in \clb(\clh)$ is said to be \textit{quasinormal} if $T^* T$ and $T$ commutes, that is
\[
[T^*, T]T = 0,
\]
where $[T^*, T] = T^* T - T T^*$ is the commutator of $T$. In this section, we present a complete classification of quasinormality of $M_z$ on analytic tridiagonal spaces. Here, however, we do not need to assume that $M_z$ is left-invertible.

To motivate our result on quasinormality, we first consider the known case of weighted shifts. Recall that the weighted shift $S_{\alpha}$ corresponding to the weight sequence (of positive real numbers) $\{\alpha_n\}_{n \geq 0}$ is given by $S_{\alpha} e_n = \alpha_n e_{n+1}$ for all $n \geq 0$. Then (see the proof of Proposition \ref{Prop:shimorin diagonal})
\[
S_{\alpha} S_{\alpha}^* e_{n+1} = \alpha_n^2 e_{n+1},
\]
and hence $(S_{\alpha}^* S_{\alpha} - S_{\alpha} S_{\alpha}^*) S_{\alpha} = 0$ if and only if $(S_{\alpha}^* S_{\alpha} - S_{\alpha} S_{\alpha}^*) S_{\alpha} e_n = 0$ for all $n\geq 0$, which is equivalent to
\[
\alpha_n (\alpha_{n+1}^2 - \alpha_n^2) = 0,
\]
for all $n$. Thus, we have proved \cite[Problem 139]{Halmos}:

\begin{Lemma}\label{Lemma: normal S_alpha}
The weighted shift $S_{\alpha}$ is quasinormal if and only if the weight sequence $\{\alpha_n\}_{n \geq 0}$ is a constant sequence.
\end{Lemma}

Now we turn to $M_z$ on a semi-analytic tridiagonal space $\clh_k$. Suppose $[M_z^*, M_z] = r P_{f_0}$, where $r$ is a non-negative real number and $P_{f_0}$ denote the orthogonal projection of $\clh_k$ onto the one dimensional space $\mathbb{C} f_0$. Then $[M_z^*, M_z]M_z = r P_{f_0} M_z$ implies that
\[
([M_z^*, M_z]M_z) f_n = r P_{f_0} (z f_n).
\]
Now by \eqref{eq: M_z} we have
\[
z f_n = \sum_{i=n+1}^\infty \beta_i f_i,
\]
for some scalar $\beta_i \in \mathbb{C}$, $i \geq n+1$. Note that $\beta_{n+1} = \frac{a_n}{a_{n+1}} \neq 0$. This shows that $P_{f_0} (z f_n) = 0$, and hence
\[
([M_z^*, M_z]M_z) f_n = 0 \quad \quad (n\geq 0),
\]
that is, $M_z$ is quasinormal. Conversely, assume that $M_z$ is a non-normal and quasinormal operator. Then $[M_z^*, M_z]M_z = 0$ implies that $\mbox{ran} M_z \subseteq \ker [M_z^*, M_z]$, and therefore, by Lemma \ref{lemma: ker Mz*}, we have
\[
\mathbb{C} f_0 = \ker M_z^* \supseteq \overline{ran} [M_z^*, M_z].
\]
Clearly this implies $[M_z^*, M_z] = r P_{f_0}$ for some non-zero scalar $r$. Then
\[
r \|f_0\|^2 = \langle r P_{f_0} f_0, f_0 \rangle_{\clh_k} =  \langle [M_z^*, M_z]f_0, f_0 \rangle_{\clh_k} = \|M_z f_0\|^2 - \|M_z^* f_0\|^2 = \|M_z f_0\|^2,
\]
as $M_z^* f_0 = 0$, which implies
\[
r = \frac{\|M_z f_0\|^2}{\|f_0\|^2} > 0.
\]
Thus, we have proved:

\begin{Theorem}\label{thm:quasinormal}
Let $\clh_k$ be a semi-analytic tridiagonal space. Assume that $M_z$ is a non-normal operator on $\clh_k$. Then $M_z$ is quasinormal if and only if there exists a positive real number $r$ such that
\[
M_z^* M_z - M_z M_z^* = r P_{f_0},
\]
where $P_{f_0}$ denote the orthogonal projection of $\clh_k$ onto the one dimensional space $\mathbb{C} f_0$.
\end{Theorem}

In more algebraic terms this result can be formulated as follows: First we recall the matrix representation of $M_z$ (see \eqref{eqn: Mz matrix})
\[
[M_z] = \begin{bmatrix}
0& 0 & 0 & 0 & \dots
\\
\frac{a_0}{a_1} & 0 & 0 & 0 & \ddots
\\
{c_0} & \frac{a_1}{a_2} & 0 & 0 & \ddots
\\
\frac{-c_0 b_2}{a_3} & c_1 & \frac{a_2}{a_3} & 0 & \ddots
\\
\frac{c_0b_2b_3}{a_3a_4} &\frac{-c_1b_3}{a_4} & c_2 & \frac{a_3}{a_4} & \ddots
\\
\frac{-c_0b_2b_3b_4}{a_3a_4a_5} &\frac{c_1b_3b_4}{a_4a_5} & \frac{-c_2b_4}{a_5} & c_3 & \ddots
\\
\vdots & \ddots & \ddots&\ddots &\ddots
\end{bmatrix}.
\]
For each $n \geq 0$, we denote by $R_n$ and $C_n$ the $n$-th row and $n$-th column, respectively, of $[M_z]$. We then identify each of these column and row vectors with elements in $\clh_k$. Then $R_n, C_n \in \clh_k$, $n \geq 0$. Using the matrix representation $[M_z^*]$ (see \eqref{eqn: Mz^* matrix}) and $[M_z]$, we get
\[
\langle R_0, R_n \rangle_{\clh_k} = 0,
\]
for all $n \geq0$, and, consequently
\[
\Big[[M_z^*, M_z]\Big]
= \begin{bmatrix}
\langle C_0, C_0 \rangle_{\clh_k} & \langle C_1,C_0 \rangle_{\clh_k} & \langle C_2,C_0 \rangle_{\clh_k} & \cdots
\\
\langle C_0 ,C_1 \rangle_{\clh_k} & \langle C_1,C_1 \rangle_{\clh_k} - \langle R_1, R_1 \rangle_{\clh_k} & \langle C_2, C_1 \rangle_{\clh_k} - \langle R_1, R_2 \rangle_{\clh_k}  & \cdots
\\
\langle C_0 ,C_2 \rangle_{\clh_k} & \langle C_1, C_2 \rangle_{\clh_k} - \langle R_2,R_1 \rangle_{\clh_k} & \langle C_2,C_2 \rangle_{\clh_k} - \langle R_2,R_2 \rangle_{\clh_k} & \cdots
\\
\vdots & \vdots & \vdots & \ddots
\end{bmatrix}.
\]
Therefore:

\begin{Corollary}
Let $\clh_k$ be a semi-analytic tridiagonal space. Then $M_z$ on $\clh_k$ is quasinormal if and only if $\langle C_0,C_0\rangle_{\clh_k} = r$ and
\[
\langle C_0, C_i\rangle_{\clh_k} = 0 \quad \quad (i \geq 1),
\]
and
\[
\langle C_n,C_m \rangle_{\clh_k} - \langle R_m,R_n \rangle_{\clh_k} =0,
\]
for all $1 \leq m \leq n$.
\end{Corollary}

It is easy to see that a quasinormal operator is always subnormal \cite{Halmos}. However, a complete classification of subnormality of $M_z$ on tridiagonal spaces is rather more subtle and not quite as clear-cut as in the quasinormal situation. In fact the general classification of subnormality of $M_z$ on tridiagonal spaces is not known (however, see \cite{Adam 2013}).

\newsection{Aluthge transforms of shifts}\label{sect: AT of Mz}

Recall that the \textit{Aluthge transform} of an operator $T \in \clb(\clh)$ is the bounded linear operator
\[
\tT = |T|^{\frac{1}{2}} U |T|^{\frac{1}{2}}.
\]
In this section, we prove that the Aluthge transform of a left-invertible shift on an analytic Hilbert space is again an explicit shift on some analytic Hilbert space. We present two approaches to this problem, one based on Shimorin's analytic models of left-invertible operators and one is based on rather direct reproducing kernel Hilbert space techniques.

We begin with the following simple fact concerning Aluthge transforms of left-invertible operators:

\begin{Lemma}\label{lemma: tilde T}
If $T$ is a left-invertible operator on $\clh$, then
\[
\tT = |T|^{\frac{1}{2}} T |T|^{- \h},
\]
and $\ker \tT^* = |T|^{-\h} \ker T^*$. In particular, $\tT$ is similar to $T$.
\end{Lemma}
\begin{proof}
Indeed, $\tT = |T|^{\h} U |T|^{\h} = |T|^{\h} (U |T|) |T|^{-\h} = |T|^{\frac{1}{2}} T |T|^{- \h}$, as $T^*T$ is invertible. The second equality follows from the first.
\end{proof}

Suppose in addition that $T$ is a shift on an analytic Hilbert space. In Theorem \ref{thm: left analytic model} (under an additional assumption that $T$ is analytic), and then in Theorem \ref{thm: revisit model} again, we prove that $\tT$, up to unitary equivalence, is also a shift on an explicit analytic Hilbert space. In connection with Lemma \ref{prop: L tilde T}, we now prove the following:

\begin{Proposition}\label{prop: L tilde T new}
If $T$ is a left-invertible operator on $\clh$, then the Shimorin left inverse $L_{\tT}$ of the Aluthge transform $\tT$ is given by
\[
L_{\tT} = |T|^{\h} \Big((L_T |T| T)^{-1} L_T\Big) |T|^{\h} = |T|^{\h} \Big((T^* |T| T)^{-1} T^*\Big) |T|^{\h}.
\]
\end{Proposition}
\begin{proof}
Note that by Lemma \ref{lemma: tilde T}, we have $\tT^* \tT = |T|^{-\h} (T^* |T| T) |T|^{-\h}$. Since $T^* |T| T$ is invertible, it follows that $(\tT^* \tT)^{-1} = |T|^{\h} (T^* |T| T)^{-1} |T|^{\h}$. Then
\[
L_{\tT} = (\tT^* \tT)^{-1} \tT^* = (|T|^{\h} (T^* |T| T)^{-1} |T|^{\h}) |T|^{-\h} T^* |T|^{\h} = |T|^{\h} \Big((T^* |T| T)^{-1} T^*\Big) |T|^{\h}.
\]
On the other hand, since $T^* = |T|^2 L_T$, we have $T^* |T| T = |T|^2 L_T |T| T$, and hence
\[
(T^* |T| T)^{-1} = (L_T |T| T)^{-1} |T|^{-2}.
\]
Therefore, $(\tT^* \tT)^{-1} = |T|^{\h} (L_T |T| T)^{-1} |T|^{-\frac{3}{2}}$, which gives
\[
L_{\tT} = (\tT^* \tT)^{-1} \tT^* = |T|^{\h} (L_T |T| T)^{-1} |T|^{-2} (T^* |T|^{\h}) = |T|^{\h} (L_T |T| T)^{-1} L_T |T|^{\h},
\]
and completes the proof.
\end{proof}

Then the above, along with Theorem \ref{thm-Shimorin} and Lemma \ref{lemma: tilde T} implies the following:

\begin{Theorem}\label{thm: left analytic model}
Let $\cle$ be a Hilbert space, and let $k : \D \times \D \raro \clb(\cle)$ be an analytic kernel. Suppose $M_z$ is left-invertible on $\clh_k$. Then the Aluthge transform $\tilde{M_z}$ is unitarily equivalent to the shift $M_z$ on $\clh_{\tilde{k}} \subseteq \clo(\D, \tilde \clw)$, where
\[
\tilde{k}(z, w) = P_{\tilde \clw} (I - z {L})^{-1} (I - \bar{w} {L}^*)^{-1}|_{\tilde \clw} \quad \quad (z, w \in \D),
\]
and $\tilde \clw = \ker \tilde M_z^* = |M_z|^{-\h} \ker M_z^*$, and
\[
{L} = |M_z|^{\h} ((L_{M_z} |M_z| M_z)^{-1} L_{M_z}) |M_z|^{\h}.
\]
\end{Theorem}

\begin{Definition}
The kernel $\tilde{k}$ is called the \textit{Shimorin-Aluthge kernel} of ${M}_z$.
\end{Definition}

Under some additional assumptions on scalar-valued analytic kernels, we now prove that, up to similarity and a perturbation of an operator of rank at most one, $L_{\tilde{M}_z}$ and $L_{M_z}$ are the same. As far as concrete examples are concerned, these assumptions are indispensable and natural (cf. Lemma \ref{lemma: ker Mz*}).

\begin{Theorem}\label{thm: rank one perturb}
Let $k : \D \times \D \raro \mathbb{C}$ be an analytic kernel, $\mathbb{C}[z] \subseteq \clh_k$, and let $\{f_n\} \subseteq \mathbb{C}[z]$ be an orthonormal basis of $\clh_k$. Assume that $M_z$ on $\clh_k$ is left-invertible, $\ker M_z^* = \mathbb{C} f_0$, and
\[
f_n \in \mbox{span} \{z^m : m \geq 1\} \qquad (n \geq 1).
\]
Then $L_{\tilde{M}_z}$ and $L_{M_z}$ are similar up to the perturbation of an operator of rank at most one.
\end{Theorem}
\begin{proof}
Since $\ker M_z^* = \mathbb{C} f_0$, $L_{M_z} f_0 = 0$ and $L_{M_z} z^n = L_{M_z} M_z (z^{n-1}) = z^{n-1}$, by the definition of $L_{M_z}$. This implies $L_{M_z} z^n = z^{n-1}$, $n\geq 1$ (also see \eqref{eqn:L_Mz back shift}). In particular, $L_{M_z} f_n \in \mathbb{C}[z]$ for all $n \geq 0$. Moreover, for each $n \geq 1$, we have
\[
\begin{split}
L_{\tilde{M}_z} (|M_z|^{\h} z^n) & = |M_z|^{\h} ((L_{M_z} |M_z| M_z)^{-1} L_{M_z}) |M_z| z^n
\\
& = |M_z|^{\h} (L_{M_z} |M_z| M_z)^{-1} (L_{M_z} |M_z| M_z) z^{n-1},
\end{split}
\]
that is, $L_{\tilde{M}_z} (|M_z|^{\h} z^n)= |M_z|^{\h} z^{n-1}$. Therefore, we have
\[
(|M_z|^{-\h} L_{\tilde{M}_z} |M_z|^{\h}) z^n = L_{M_z} z^n = z^{n-1} \quad \quad (n\geq 1).
\]
Then $(|M_z|^{-\h} L_{\tilde{M}_z} |M_z|^{\h} - L_{M_z})f_n = 0$ for all $n \geq 1$, which gives
\[
(|M_z|^{-\h} L_{\tilde{M}_z} |M_z|^{\h} - L_{M_z})|_{\overline{span}\{f_n: n \geq 1\}} = 0.
\]
Finally, we have clearly $(|M_z|^{-\h} L_{\tilde{M}_z} |M_z|^{\h} - L_{M_z})f_0 = (|M_z|^{-\h} L_{\tilde{M}_z} |M_z|^{\h}) f_0$, and hence
\begin{equation}\label{eqn:F = rank one 1}
F:= |M_z|^{-\h} L_{\tilde{M}_z} |M_z|^{\h} - L_{M_z},
\end{equation}
is of rank at most one, and consequently $L_{\tilde{M}_z} |M_z|^{\h} = |M_z|^{\h}(L_{M_z} + F)$. This completes the proof of the theorem.
\end{proof}

The following analysis of $F$, defined as in \eqref{eqn:F = rank one 1}, will be useful in what follows. Note that
\begin{equation}\label{eqn:tilde L similar L+F}
L_{\tilde{M}_z} |M_z|^{\h} = |M_z|^{\h}(L_{M_z} + F).
\end{equation}
Let $g \in \clh_k$. Clearly, since $L_{M_z} f_0 = 0$, we have $F g = \BL g, f_0 \BR_{\clh_k} (|M_z|^{-\h} L_{\tilde{M}_z} |M_z|^{\h} f_0)$. Then Lemma \ref{prop: L tilde T} implies that
\begin{equation}\label{eqn:F = rank one 2}
F g = \BL g, f_0 \BR_{\clh_k} ((M_z^* |M_z| M_z)^{-1} M_z^* |M_z| f_0) \quad \quad (g \in \clh_k).
\end{equation}
As we will see in Section \ref{sec: truncated}, the appearance of the finite rank operator $F$ causes severe computational difficulties for Shimorin-Aluthge kernels of shifts. On the other hand, combining Theorem \ref{thm-Shimorin}, Proposition \ref{prop: L tilde T new} and \eqref{eqn:tilde L similar L+F}, we have:

\begin{Theorem}\label{thm: left analytic model version 2}
In the setting of Theorem \ref{thm: rank one perturb}, the Aluthge transform $\tilde{M}_z$ of $M_z$ on $\clh_k$ is unitarily equivalent to the shift $M_z$ on $\clh_{\tilde{k}}$, where
\[
\tilde{k}(z,w) = P_{\clw} (I - z L)^{-1} (I - \bar{w} L^*)^{-1}|_{\clw},
\]
$\clw = |M_z|^{-\h} \ker M_z^* = \mathbb{C} (|M_z|^{-\h} f_0)$, and
\[
L = |M_z|^{\h} (L_{M_z} + F) |M_z|^{-\h},
\]
and $F g = \BL g, f_0 \BR_{\clh_k} ((M_z^* |M_z| M_z)^{-1} M_z^* |M_z| f_0)$ for all $g \in \clh_k$.
\end{Theorem}

We now revisit Theorem \ref{thm: left analytic model} from a direct reproducing kernel Hilbert space standpoint. Indeed, there is a rather more concrete proof of Theorem \ref{thm: left analytic model} which avoids using the analytic model of left-invertible operators. In this case, also, the reproducing kernel of the corresponding Aluthge transform is explicit. Part of the proof follows the same line of argumentation as the proof of reproducing kernel property of range spaces (cf. \cite{Paulsen 92}). To the reader's benefit, we include all necessary details.

\begin{Theorem}\label{thm: revisit model}
Let $\cle$ be a Hilbert space, and let $k : \D \times \D \raro \clb(\cle)$ be an analytic kernel. Assume that the shift $M_z$ is left-invertible on $\clh_k$. Then
\[
\langle \tilde{k}(z, w) \eta, \zeta \rangle_{\cle} = \langle |M_z|^{-1} (k(\cdot, w) \eta), k(\cdot, z) \zeta) \rangle_{\clh_k} \quad \quad (z, w \in \D, \eta, \zeta \in \cle),
\]
defines a kernel $\tilde{k} : \D \times \D \raro \clb(\cle)$. Moreover, the shift $M_z$ on $\clh_{\tilde{k}}$ defines a bounded linear operator, and there exists a unitary $U : \clh_k \raro \clh_{\tilde{k}}$ such that $U \tilde{M}_z = M_z U$.
\end{Theorem}
\begin{proof}
Define $\tilde \clh = |M_z|^{-\h} \clh_k$. Then $\tilde{\clh} (= \clh_k)$ is an $\cle$-valued function Hilbert space endowed with the inner product $\langle |M_z|^{-\h} f, |M_z|^{-\h} g \rangle_{\tilde{\clh}} = \langle f, g \rangle_{\clh_k}$ for all $f, g \in \clh_k$. For each $f \in \clh_k$, $w \in \D$ and $\eta \in \cle$, we have
\[
\langle |M_z|^{-\h} f, |M_z|^{-1} (k(\cdot, w) \eta) \rangle_{\tilde{\clh}}  = \langle f,  |M_z|^{-\h} (k(\cdot, w)\eta) \rangle_{\clh_k} = \langle |M_z|^{-\h} f, k(\cdot, w) \eta \rangle_{\clh_k},
\]
and hence, by the reproducing property of $\clh_k$, it follows that
\begin{equation}\label{eqn: thm 2.8 kernel equality}
\langle |M_z|^{-\h} f, |M_z|^{-1} (k(\cdot, w) \eta) \rangle_{\tilde{\clh}} = \langle (|M_z|^{-\h} f) (w), \eta \rangle_{\cle}.
\end{equation}
This says that $\{|M_z|^{-1} (k(\cdot, w) \eta): w \in \D, \eta \in \cle\}$ reproduces the values of functions in $\tilde{\clh}$, and furthermore, the evaluation operator $ev_w : \tilde \clh \raro \cle$ is continuous. Indeed
\[
\begin{split}
|\langle ev_w (|M_z|^{-\h} f), \eta \rangle_{\cle} | & = |\langle (|M_z|^{-\h} f) (w), \eta \rangle_{\cle}|
\\
& = |\langle |M_z|^{-\h} f, |M_z|^{-1} (k(\cdot, w) \eta) \rangle_{\tilde{\clh}}|
\\
& \leq \||M_z|^{-\h} f\|_{\tilde{\clh}} \||M_z|^{-1} (k(\cdot, w) \eta)\|_{\tilde{\clh}}
\\
& = \||M_z|^{-\h} f\|_{\tilde{\clh}} \; \||M_z|^{-\h} (k(\cdot, w) \eta)\|_{\clh_k}.
\end{split}
\]
Since $\|k(\cdot, w) \eta\|^2_{\clh_k} = \langle k(\cdot, w) \eta, k(\cdot, w) \eta \rangle_{\clh_k} = \langle k(w, w) \eta, \eta \rangle_{\cle} = \|k(w, w)^{\frac{1}{2}} \eta\|^2_{\cle}$, it follows that
\[
\begin{split}
\||M_z|^{-\h} (k(\cdot, w) \eta)\|_{\clh_k} & \leq \| |M_z|^{-\h}\|_{\clb(\clh_k)} \|k(\cdot, w) \eta\|_{\clh_k}
\\
& = \| |M_z|^{-\h}\|_{\clb(\clh_k)}\; \|k(w, w)^{\frac{1}{2}} \eta\|_{\cle}
\\
& \leq \| |M_z|^{-\h}\|_{\clb(\clh_k)}\; \|k(w, w)^{\frac{1}{2}}\|_{\clb(\cle)}\; \|\eta\|_{\cle},
\end{split}
\]
which implies that
\[
| \langle ev_w (|M_z|^{-\h} f), \eta \rangle_{\cle}| \leq (\| |M_z|^{-\h}\|_{\clb(\clh_k)} \|k(w, w)^{\frac{1}{2}}\|_{\clb(\cle)})  \||M_z|^{-\h} f\|_{\tilde{\clh}}\; \|\eta\|_{\cle}.
\]
Therefore $\tilde{\clh}$ is an $\cle$-valued reproducing kernel Hilbert space corresponding to the kernel function
\[
\tilde{k}(z, w) = ev_z \circ ev_w^* \quad \quad (z, w \in \D).
\]
Clearly, \eqref{eqn: thm 2.8 kernel equality} implies that $ev_w^* \eta = |M_z|^{-1} (k(\cdot, w) \eta)$ for all $w \in \D$ and $\eta \in \cle$. Since $\langle \tilde{k}(z, w) \eta, \zeta \rangle_{\cle}  = \langle ev_w^* \eta, ev_z^* \zeta \rangle_{\cle}$, it follows that
\[
\begin{split}
\langle \tilde{k}(z, w) \eta, \zeta \rangle_{\cle} & = \langle |M_z|^{-1} (k(\cdot, w) \eta), |M_z|^{-1} (k(\cdot, z) \zeta) \rangle_{\tilde{\clh}}
\\
& = \langle |M_z|^{-\h} (k(\cdot, w) \eta), |M_z|^{-\h} (k(\cdot, z) \zeta)\rangle_{\clh_k},
\end{split}
\]
that is, $\langle \tilde{k}(z, w) \eta, \zeta \rangle_{\cle} = \langle |M_z|^{-1} (k(\cdot, w) \eta), k(\cdot, z) \zeta) \rangle_{\clh_k}$, $z, w \in \D, \eta, \zeta \in \cle$. Therefore, as a reproducing kernel Hilbert space corresponding to the kernel $\tilde k$, we have $\clh_{\tilde{k}} = \tilde \clh$. Define the unitary map $U : \clh_k \raro \clh_{\tilde{k}}$ by
\[
Uh = |M_z|^{-\frac{1}{2}}h \quad\quad (h \in \clh_k),
\]
and recall from  Lemma \ref{lemma: tilde T} that $\tilde{M}_z^* = |M_z|^{-\h} M_z^* |M_z|^{\h}$. Let $f \in \clh_k$, $w \in \D$, and let $\eta \in \cle$. Then
\[
\begin{split}
\langle (U \tilde{M}_z U^* (|M_z|^{-\h} f))(w), \eta \rangle_{\cle} & = \langle U \tilde{M}_z U^* (|M_z|^{-\h} f), |M_z|^{-1} (k(\cdot, w) \eta) \rangle_{\clh_{\tilde k}}
\\
& = \langle \tilde{M}_z U^* (|M_z|^{-\h} f), |M_z|^{-\h} (k(\cdot, w) \eta) \rangle_{\clh_{k}}
\\
& = \langle f, \tilde{M}_z^* |M_z|^{-\h} (k(\cdot, w) \eta) \rangle_{\clh_{k}}
\\
& = \langle f, |M_z|^{-\h}  {M}_z^* (k(\cdot, w) \eta) \rangle_{\clh_{k}}.
\end{split}
\]
But since ${M}_z^* (k(\cdot, w) \eta) = \bar{w} k(\cdot, w) \eta$, we have
\[
\langle (U \tilde{M}_z U^* (|M_z|^{-\h} f))(w), \eta \rangle_{\cle} = w \langle f, |M_z|^{-\h} (k(\cdot, w) \eta) \rangle_{\clh_{k}} = \langle w (|M_z|^{-\h} f\big))(w), \eta \rangle_{\cle},
\]
which implies that $U \tilde{M}_z U^* (|M_z|^{-\h} f) = z  (|M_z|^{-\h} f)$ for all $f \in \clh_k$. Thus the shift $M_z$ on $\clh_{\tilde k}$ is a bounded linear operator and $U \tilde{M}_z = M_z U$.
\end{proof}

\begin{Definition}
The kernel $\tilde{k}$ is called the \textit{standard Aluthge kernel} of $M_z$.
\end{Definition}

In particular, if $k$ is a scalar-valued kernel, then $\tilde{k}(\cdot, w) = U (|M_z|^{-\frac{1}{2}} k(\cdot, w))$ and
\[
\tilde{k} (z, w) = \BL |M_z|^{-1} {k}(\cdot, w), {k}(\cdot, z) \BR_{\clh_k} \quad \quad (z, w\in \D).
\]
Therefore, if the shift on a tridiagonal space $\clh_k$ is left-invertible, then there are two ways to compute the Aluthge kernel $\tilde{k}$: use Theorem \ref{thm: left analytic model}, or use the one above. However, it is curious to note that, from a general computational point of view, neither approach is completely satisfactory and definite. On the other hand, often the standard Aluthge kernel approach (and sometimes both standard Aluthge kernel and Shimorin-Aluthge kernel methods) lead to satisfactory results. We will discuss this in the following section.

\newsection{Truncated tridiagonal kernels}\label{sec: truncated}

In this section, we introduce a (perhaps both deliberate and accidental) class of analytic tridiagonal kernels from a computational point of view. Let $\clh_k$ be an analytic tridiagonal space corresponding to the kernel
\[
k(z, w) = \sum_{n=0}^\infty f_n(z) \overline{f_n(w)} \quad \quad (z, w \in \D),
\]
where $f_n = (a_n + b_n z) z^n$, $n \geq 0$.  Suppose $r \geq 2$ is a natural number. We say that $k$ is a \textit{truncated tridiagonal kernel of order $r$} (in short, \textit{truncated kernel of order $r$}) if
\[
b_n = 0 \quad \quad (n \neq 2, 3, \ldots, r).
\]
We say that an analytic tridiagonal space $\clh_k$ is \textit{truncated space of order $r$} if $k$ is a truncated kernel of order $r$. Note that there are no restrictions imposed on the scalars $b_2, \ldots, b_r$.

Let $\clh_k$ be a truncated space of order $r$. Then $\tilde{M}_z$ is unitarily equivalent to $M_z$ on $\clh_{\tilde{k}}$, where $\tilde{k}$ is either the Shimorin-Aluthge kernel or the standard Aluthge kernel of $M_z$ as in Theorem \ref{thm: left analytic model} and Theorem \ref{thm: revisit model}, respectively. Here our aim is to compute the Shimorin-Aluthge kernel of ${M}_z$. More specifically, we classify all truncated kernels $k$ such that the Shimorin-Aluthge kernel $\tilde{k}$ of ${M}_z$ is tridiagonal. We begin by computing $|M_z|^{-1}$.

\begin{Lemma}\label{lemma: truncated}
If $\clh_k$ is a truncated space of order $r$, then
\[
\Big[|M_z|^{-1}\Big] =
\begin{bmatrix}
|\frac{a_1}{a_0}| & 0 & 0 &\cdots & 0 & 0 & 0 & \cdots
\\
0 & c_{11} & c_{12} &\cdots & c_{1, r+1} & 0 & 0 & \ddots
\\
0 & \bar{c}_{12} & c_{22} & \cdots & c_{2, r+1} & 0 & 0 & \ddots
\\
\vdots &\vdots & \vdots & \cdots & \vdots & \vdots & \vdots & \ddots
\\
0 & \bar{c}_{1, r+1} & \bar{c}_{2, r+1} &\cdots & c_{r+1, r+1} & 0 & 0 & \ddots
\\
0 & 0 & 0 &\cdots & 0 & |\frac{a_{r+3}}{a_{r+2}}|& 0 & \ddots
\\
0 & 0 & 0 &\cdots & 0 & 0 & |\frac{a_{r+4}}{a_{r+3}}|&  \ddots
\\
\vdots & \vdots & \vdots &\cdots & \vdots & \vdots & \ddots &\ddots
\end{bmatrix},
\]
with respect to the orthonormal basis $\{f_n\}_{n \geq 0}$.
\end{Lemma}
\begin{proof}
For each $n \geq 1$, by the definition of $d_n$ from \eqref{eqn:d_n}, we have $d_n = \frac{b_n}{a_n} - \frac{b_{n-1}}{a_{n-1}}$, and hence $d_1 = d_{r+i} = 0 $, $i=2, 3, \ldots$. Then Theorem \ref{Prop: L Matrix} tells us that
\[
[L_{M_z}] = \begin{bmatrix}
0 & \frac{a_1}{a_0} & 0  & \cdots & 0  & 0 & 0 & 0 & \cdots
\\
0 & 0 & \frac{a_2}{a_1}  & \cdots & 0  & 0 & 0 & 0 & \ddots
\\
0 & 0 & d_2  & \cdots & 0  & 0 & 0 & 0 & \ddots
\\
\vdots & \vdots & \vdots & \vdots & \vdots & \vdots & \vdots & \vdots & \ddots
\\
0 & 0 &(-1)^{r-2}\frac{d_2 b_2\cdots b_{r-1}}{a_3\cdots a_r} & \cdots & d_r  & \frac{a_{r+1}}{a_r} & 0  & 0  & \ddots
\\
0 & 0 &(-1)^{r-1}\frac{d_2 b_2\cdots b_r}{a_3\cdots a_r a_{r+1}} & \cdots & -\frac{d_r b_r}{a_{r+1}} &  d_{r+1} & \frac{a_{r+2}}{a_{r+1}} & 0 & \ddots
\\
0 & 0 & 0  & \cdots & 0  & 0 & 0 & \frac{a_{r+3}}{a_{r+2}} &\ddots
\\
\vdots & \vdots & \vdots & \vdots & \vdots & \vdots & \vdots & \ddots & \ddots
\end{bmatrix}.
\]
Now, by Lemma \ref{prop: L tilde T}, $|M_z|^{-2} = L_{M_z} L^*_{M_z}$, which implies
\[
\Big[|M_z|^{-2}\Big] = \begin{bmatrix}
|\frac{a_1}{a_0}|^2 & 0 & 0
\\
0 & A^2_{r+1} & 0
\\
0 & 0 & D^2
\end{bmatrix},
\]
where $D^2 = \diag \Big(\Big|\frac{a_{r+3}}{a_{r+2}}\Big|^2, \Big|\frac{a_{r+4}}{a_{r+3}}\Big|^2, \ldots \Big)$ and $A^2_{r+1}$ is a positive definite matrix of order $r+1$. Using this, one easily completes the proof.
\end{proof}

From the computational point of view, it is useful to observe that $A^2_{r+1} = L_{r+1} L_{r+1}^*$, where
\[
L_{r+1} = \begin{bmatrix}
\frac{a_2}{a_1}  & 0  & 0 & 0  & 0
\\
d_2  & \frac{a_3}{a_2}  & 0 & 0  & 0
\\
\vdots & \vdots & \vdots & \vdots & \vdots
\\
(-1)^{r-2}\frac{d_2 b_2\cdots b_{r-1}}{a_3\cdots a_r} &  (-1)^{r-3}\frac{d_3 b_3 \cdots b_{r-1}}{a_4 \cdots a_r} & \cdots & \frac{a_{r+1}}{a_r} & 0
\\
(-1)^{r-1}\frac{d_2 b_2\cdots b_r}{a_3\cdots a_r a_{r+1}} & (-1)^{r-2}\frac{d_3 b_3\cdots b_r}{a_4\cdots a_r a_{r+1}} & \cdots  &  d_{r+1} & \frac{a_{r+2}}{a_{r+1}}
\end{bmatrix}.
\]
In other words, $A^2_{r+1}$ admits a lower-upper triangular factorization. This is closely related to the Cholesky factorizations/decompositions of positive-definite matrices in the setting of infinite dimensional Hilbert spaces (see \cite{Paulsen 92} and \cite{Woerdeman}).

We recall from Theorem \ref{thm: left analytic model version 2} that the Shimorin-Aluthge kernel of ${M}_z$ is given by
\[
\tilde{k}(z, w) = P_{\tilde \clw} (I - z L_{\tilde M_z})^{-1} (I - \bar{w} L_{\tilde M_z}^*)^{-1}|_{\tilde \clw} \quad \quad (z, w \in \D),
\]
where $\tilde \clw = |M_z|^{-\h} \ker M_z^*$, and
\begin{equation}\label{eqn: truncated L tilde M}
L_{\tilde M_z} = |M_z|^{\h} (L_{M_z} + F) |M_z|^{-\h},
\end{equation}
and $F g = \BL g, f_0 \BR_{\clh_k} \Big((M_z^* |M_z| M_z)^{-1} M_z^* |M_z| f_0\Big)$, $g \in \clh_k$. We now come to the key point.

\begin{Lemma}\label{lemma: truncated 1}
If $k$ is a truncated kernel, then $F = 0$ and ${L}_{\tilde{M}_z} |M_z|^{\h} = |M_z|^{\h} L_{M_z}$.
\end{Lemma}
\begin{proof}
The matrix representation of $|M_z|^{-1}$ in Lemma \ref{lemma: truncated} implies that $|M_z| f_0 = |\frac{a_0}{a_1}| f_0$, and hence
\[
M_z^* |M_z| f_0 = \Big|\frac{a_0}{a_1}\Big| M_z^* f_0 = 0,
\]
by Lemma \ref{lemma: ker Mz*}. Therefore, the proof follows from the definition of $F$ and \eqref{eqn: truncated L tilde M}.
\end{proof}

We are finally ready to state and prove the result we are aiming for.

\begin{Theorem}\label{thm: truncated}
Let $\clh_k$ be a truncated space of order $r$. Then the Shimorin-Aluthge kernel is tridiagonal if and only if
\[
c_{mn}=(-1)^{n-m-1} \frac{\bar{b}_{m+1} \cdots \bar{b}_{n-1}}{\bar{a}_{m+2} \cdots \bar{a}_n} c_{m, m+1},
\]
for all $1 \leq m \leq n-2$ and $3 \leq n \leq r+1$, where $c_{mn}$ are the entries of the middle block submatrix of order $r+1$ of $\Big[ |M_z|^{-1}\Big]$ in Lemma \ref{lemma: truncated}.
\end{Theorem}
\begin{proof}
We split the proof into several steps.

\noindent \textsf{Step 1:} First observe that $\tilde{k}(z,w) = \sum_{m,n=0}^\infty \tilde{X}_{mn} z^m \bar{w}^n$, where $\tilde{X}_{mn} = P_{\tilde \clw} L_{\tilde{M}_z}^{m} L_{\tilde{M}_z}^{*n}|_{\tilde \clw}$ for all $m,n \geq 0$. Now Lemma \ref{lemma: truncated 1} implies that ${L}_{\tilde{M}_z}^m {L}_{\tilde{M}_z}^{*n} = |M_z|^{\h} L_{M_z}^m |M_z|^{-1} L_{M_z}^{*n} |M_z|^{\h}$, and $P_{\tilde{\clw}} = I - \tilde{M}_z {L}_{\tilde{M}_z}$ by \eqref{eqn:P_W = I- TL}. Since $\tilde{M}_z = |M_z|^{\h} M_z |M_z|^{-\h}$ and ${L}_{\tilde{M}_z} = |M_z|^{\h} L_{M_z} |M_z|^{-\h}$, we have
\[
P_{\tilde \clw} = |M_z|^{\h} (I - M_z L_{M_z}) |M_z|^{-\h} = |M_z|^{\h} P_{\clw} |M_z|^{-\h}
\]
that is, $P_{\tilde \clw} |M_z|^{\h} = |M_z|^{\h} P_{\clw}$, which implies
\begin{equation}\label{eqn:tilde{X}_{mn}}
\tilde{X}_{mn} = |M_z|^{\h} P_{\clw}  L_{M_z}^m |M_z|^{-1} L_{M_z}^{*n}|_{\clw} \quad \quad (m,n \geq 0).
\end{equation}
As a passing remark, we note that the above equality holds so long as the finite rank operator $F = 0$ (this observation also will be used in Example \ref{example: truncated SK not SK}).

\noindent \textsf{Step 2:} Now we compute the matrix representation of $L^p_{M_z}$, $p \geq 1$. By Theorem \ref{Prop: L Matrix}, we have

\[
[L_{M_z}] = \begin{bmatrix}
0 & \frac{a_1}{a_0} & 0 & 0  & 0 & \dots
\\
0 & 0 & \frac{a_2}{a_1} & 0  & 0 & \ddots
\\
0 & 0 & d_2 & \frac{a_3}{a_2} & 0  &  \ddots
\\
0 & 0 &\frac{-d_2b_2}{a_3} & d_3 &\frac{a_4}{a_3} & \ddots
\\
\vdots & \vdots & \vdots & \vdots &\vdots &\ddots
\end{bmatrix}
\]
In particular, this yields
\[
P_{\mathcal{W}} L_{M_z}f_j = \begin{cases}
\frac{a_1}{a_0}f_0 & \mbox{if } j=1 \\
0 & \mbox{otherwise}.
\end{cases}
\]
Now we let $p \geq 2$. Recall from \eqref{eqn: beta n p} the definition $\beta_n^{(p)} = a_n \Big(\frac{-b_0}{a_0}\Big)^{p-n-1} \beta_n$ for all $n = 1, \ldots, p-1$, where $\beta_n = \frac{b_n}{a_n} - \frac{b_0}{a_0}$. Since $b_0 = 0$, we have $\beta_n^{(p)} = 0$, $1 \leq n < p-1$, and
\[
\beta_{p-1}^{(p)} =  a_{p-1} \beta_{p-1} = a_{p-1} \Big(\frac{b_{p-1}}{a_{p-1}} - \frac{b_0}{a_0}\Big),
\]
that is, $\beta_{p-1}^{(p)} = b_{p-1}$ for all $p \geq 2$. In particular, since $b_1 = 0$, we have $\beta_{1}^{(2)} = b_1 = 0$. Also recall from \eqref{eqn: d^p_n} the definition $d_n^{(p)} = b_n - \frac{a_n}{a_{n-p}} b_{n-p}$, $n \geq p$. Therefore, by \eqref{eqn:[L2]}, the matrix representation of $L^2_{M_z}$ is given by
\[
[L_{M_z}^2] = \begin{bmatrix}
0 & 0 & \frac{a_2}{a_0} & 0 & 0 & \cdots
\\
0 & 0 & \frac{d_2^{(2)}}{a_1} & \frac{a_{3}}{a_1} & 0 & \ddots
\\
0 & 0 & 0 & \frac{d_{3}^{(2)}}{a_2} & \frac{a_{4}}{a_2} & \ddots
\\
\vdots & \vdots & \vdots &\ddots &\ddots
\end{bmatrix},
\]
and in general, by \eqref{eqn:Lp}, we have
\begin{equation}\label{eqn: truncated LP}
[L_{M_z}^p] = \begin{bmatrix}
0 & \cdots & 0&  \frac{b_{p-1}}{a_0} & \frac{a_p}{a_0} & 0 & 0  & \cdots
\\
0 & \cdots & 0 & 0 & \frac{d_p^{(p)}}{a_1}   & \frac{a_{p+1}}{a_1} & 0 & \ddots
\\
0 &\cdots  & 0 & 0 & 0 & \frac{d_{p+1}^{(p)}}{a_2} & \frac{a_{p+2}}{a_2} & \ddots
\\
0 & \cdots & 0  & 0 & 0 & -\frac{d_{p+1}^{(p)} b_2}{a_2 a_3} & \frac{d_{p+2}^{(p)}}{a_3} & \ddots
\\
\vdots & \vdots & \vdots &\vdots &\vdots &\vdots &\ddots &\ddots
\end{bmatrix} \quad \quad (p \geq 2).
\end{equation}
Then
\begin{equation}\label{eqn: truncated L*P}
[L_{M_z}^{*p}] = \begin{bmatrix}
0 & 0 & 0 & 0 & \cdots
\\
\vdots & \vdots & \vdots & \vdots & \ddots
\\
0 & 0 & 0 & 0 & \ddots
\\
\frac{\bar{b}_{p-1}}{\bar{a}_0} & 0 & 0 & 0 & \ddots
\\
\frac{\bar{a}_p}{\bar{a}_0} & \frac{\bar{d}_p^{(p)}}{\bar{a}_1} & 0 & 0 & \ddots
\\
0 & \frac{\bar{a}_{p+1}}{\bar{a}_1} & \frac{\bar{d}_{p+1}^{(p)}}{\bar{a}_2} & - \frac{\bar{d}_{p+1}^{(p)} \bar{b}_2}{\bar{a}_2 \bar{a}_3} &  \ddots
\\
\vdots & \vdots & \vdots &\ddots &\ddots
\end{bmatrix} \quad \quad (p \geq 2).
\end{equation}

\noindent \textsf{Step 3:}
We prove that $\tilde{X}_{0n} = |M_z|^{\h} P_{\clw} |M_z|^{-1} L_{M_z}^{*n}|_{\clw} = 0$ for all $n \geq 1$. In what follows, the above matrix representations and the one of $|M_z|^{-1}$ in Lemma \ref{lemma: truncated} will be used repeatedly. By \eqref{eqn:[L*]}, we have $L_{M_z}^{*} f_0 = \frac{\bar{a}_1}{\bar{a}_0} f_1$, and hence
\[
\tilde{X}_{01} f_0 = |M_z|^{\h} P_{\clw}|M_z|^{-1}L_{M_z}^* f_0 = |M_z|^{\h} P_{\clw} (\frac{\bar{a_1}}{\bar{a_0}}[c_{11}f_1 + \bar{c}_{12} f_2 + \cdots]) = 0.
\]
On the other hand, if $n \geq 2$, then $L_{M_z}^{*n} f_0 = \frac{\bar{b}_{n-1}}{\bar{a}_0} f_{n-1} + \frac{\bar{a}_n}{\bar{a}_0} f_n$, and hence $|M_z|^{-1} f_0 \perp L_{M_z}^{*n} f_0$. This implies that $\tilde{X}_{0n} = 0$ for all $n\geq 2$. Therefore, all entries in the first row (and hence, also in the first column) of the formal matrix representation of $\tilde{k}(z,w)$  are zero except the $(0,0)$-th entry (which is $I_{\clw}$). Hence (see also \eqref{eqn:[k Mz]})
\[
\Big[\tilde k(z,w)\Big] =
\begin{bmatrix}
I_{\tilde \clw} & 0 & 0 & 0 & \cdots
\\
0 & \tilde{X}_{11} & \tilde{X}_{12} & \tilde{X}_{13} & \ddots
\\
0 & \tilde{X}_{12}^* & \tilde{X}_{22} & \tilde{X}_{23} & \ddots
\\
0 & \tilde{X}_{13}^* & \tilde{X}_{23}^* & \tilde{X}_{33} & \ddots
\\
\vdots & \vdots & \vdots &  \ddots & \ddots
\end{bmatrix}.
\]

\noindent \textsf{Step 4:} Our only interest here is to analyze the finite rank (of rank at most one) operator $\tilde{X}_{m, m+k}$, $m \geq 1$, $k \geq 2$. The matrix representation in \eqref{eqn: truncated L*P} implies
\begin{equation}\label{eqn: truncated L * m+k}
L_{M_z}^{* m+k} f_0 = \frac{1}{\bar{a}_0} \big(\bar{b}_{m+k-1} f_{m+k-1} + \bar{a}_{m+k} f_{m+k}),
\end{equation}
and hence
\begin{equation}\label{eqn: truncated M^-1L * m+k}
|M_z|^{-1} L_{M_z}^{* m+k} f_0 = \frac{1}{\bar{a}_0} (\bar{b}_{m+k-1} |M_z|^{-1} f_{m+k-1} + \bar{a}_{m+k} |M_z|^{-1} f_{m+k}).
\end{equation}
There are three cases to be considered:

\NI\textit{Case I $(m+k =r+2)$:} Note that $b_{r+1} = 0$. Then $|M_z|^{-1} L_{M_z}^{* r+2} f_0 = \frac{1}{\bar{a}_0} (\bar{a}_{r+2} |M_z|^{-1} f_{r+2})$, by \eqref{eqn: truncated M^-1L * m+k}, and thus
\[
L_{M_z}^m |M_z|^{-1} L_{M_z}^{* r+2} f_0  = \frac{\bar{a}_{r+2}}{\bar{a}_0} L^m_{M_z} |M_z|^{-1} f_{r+2} = \frac{\bar{a}_{r+2}}{\bar{a}_0} \Big| \frac{a_{r+3}}{a_{r+2}}\Big| L^m_{M_z} f_{r+2}.
\]
By \eqref{eqn: truncated LP}, we have $P_{\clw} L^m_{M_z} f_{r+2} = P_{\clw} L^m_{M_z} f_{m+k} = 0$ (note that $k\ge 2$), and hence
\[
P_{\clw} L_{M_z}^m |M_z|^{-1} L_{M_z}^{* r+2} f_0 = 0,
\]
that is, $\tilde{X}_{m, m+k} = 0$. It is easy to check that the equality also holds for $m=1$.

\NI\textit{Case II $(m+k-1 \geq r+2)$:} In this case, $b_{m+k-1}=0$ and
\[
|M_z|^{-1} f_{m+k} = \Big| \frac{a_{m+k+1}}{a_{m+k}}\Big| f_{m+k}.
\]
Again, by \eqref{eqn: truncated LP}, we have $P_{\clw} L^m_{M_z}  f_{m+k} = 0$, $k \geq 2$, and hence in this case also $\tilde{X}_{m, m+k} = 0$. Again, it is easy to check that the equality holds for $m=1$.

\NI\textit{Case III $(m+k < r+2)$:} We again stress that $m \geq 1$ and $k \geq 2$. It is useful to observe, by virtue of \eqref{eqn: truncated LP} (also see \eqref{eqn:L*n fj}), that
\[
P_{\clw} L^m_{M_z} f_j =
\begin{cases} \frac{b_{m-1}}{a_0} f_0 & \mbox{if } j=m-1
\\
\frac{a_{m}}{a_0} f_0 & \mbox{if } j=m
\\
0 & \mbox{otherwise}.
\end{cases}
\]
Now set $s = m+k-1$. The matrix representation of $|M_z|^{-1}$ in Lemma \ref{lemma: truncated} implies that
\[
|M_z|^{-1} f_{s} = c_{1 s} f_1 + c_{2 s} f_2 + \cdots + c_{ss} f_s + \bar{c}_{s, s+1} f_{s+1} + \cdots + \bar{c}_{s, r+1} f_{r+1}.
\]
By \eqref{eqn: truncated LP} and the above equality, we have
\begin{equation}\label{eqn: page 39; 1}
P_{\clw} L_{M_z}^m |M_z|^{-1} f_s = ( c_{m-1, s}  \frac{b_{m-1}}{a_0} + c_{m, s} \frac{a_{m}}{a_0}) f_0.
\end{equation}
Next, set $t = m+k$. Again, the matrix representation of $|M_z|^{-1}$ in Lemma \ref{lemma: truncated} implies that
\[
|M_z|^{-1} f_{t} = c_{1 t} f_1 + c_{2 t} f_2 + \cdots + c_{tt} f_t + \bar{c}_{t, t+1} f_{t+1} + \cdots + \bar{c}_{t, r+1} f_{r+1},
\]
and, again, by \eqref{eqn: truncated LP} and the above equality, we have
\begin{equation}\label{eqn: page 39; 2}
P_{\clw} L_{M_z}^m |M_z|^{-1} f_t = ( c_{m-1, t}  \frac{b_{m-1}}{a_0} + c_{m, t} \frac{a_{m}}{a_0}) f_0.
\end{equation}
It is easy to see that the equalities \eqref{eqn: page 39; 1} and \eqref{eqn: page 39; 2} also holds for $m=1$. The equality in \eqref{eqn: truncated L * m+k} becomes
\[
|M_z|^{-1} L_{M_z}^{* m+k} f_0 = \frac{1}{\bar{a}_0}(\bar{b}_{s} |M_z|^{-1} f_s + \bar{a}_t |M_z|^{-1} f_t),
\]
and hence, the one in \eqref{eqn: truncated M^-1L * m+k} implies
\[
P_{\clw} L_{M_z}^m |M_z|^{-1} L_{M_z}^{* m+k} f_0 = \frac{1}{|a_0|^2} [ \bar{b}_s (c_{m-1, s}  b_{m-1} + c_{m, s} a_m) + \bar{a}_t (c_{m-1, t}  b_{m-1} + c_{m, t} a_{m})] f_0.
\]
This shows that $P_{\clw} L_{M_z}^m |M_z|^{-1} L_{M_z}^{* m+k} f_0 = 0$  if and only if
\[
\bar{b}_s (c_{m-1, s}  b_{m-1} + c_{m, s} a_m) + \bar{a}_t (c_{m-1, t}  b_{m-1} + c_{m, t} a_{m}) = 0.
\]

\noindent \textsf{Step 5:} So far all we have proved is that $\tilde{k}$ is tridiagonal if and only if
\begin{equation}\label{eqn: truncated 0 condition}
b_{m-1} (\bar{b}_{m+k-1} c_{m-1, m+k-1} + \bar{a}_{m+k} c_{m-1, m+k}) + a_m (\bar{b}_{m+k-1} c_{m, m+k-1} + \bar{a}_{m+k} c_{m, m+k}) = 0,
\end{equation}
for all $m \geq 1$, $k \geq 2$ and $m+k < r+2$.

\NI If $m=1$, then using the fact that $b_0=0$, we have $c_{1, k+1} = -\frac{\bar{b}_k}{\bar{a}_{1+k}} c_{1, k}$, $2 \leq k < r+1$, and hence
\[
c_{1 n} = (-1)^{n-2} \frac{\prod_{i=2}^{n-1} \bar{b}_i}{\prod_{i=3}^n \bar{a}_i} c_{12} \quad \quad (3 \leq n \leq r+1).
\]
Similarly, if $m=2$, then \eqref{eqn: truncated 0 condition} together with the assumption that $b_1=0$ implies that
\begin{equation}\label{eqn: truncated c 2n}
c_{2n} = (-1)^{n-3} \frac{\prod_{i=3}^{n-1} \bar{b}_i}{\prod_{i=4}^n \bar{a}_i} c_{23} \quad \quad (4 \leq n \leq r+1).
\end{equation}
Next, if $m=3$, then \eqref{eqn: truncated 0 condition} again implies
\[
b_2(\bar{b}_{k+2} c_{2, k+2} + \bar{a}_{k+3} c_{2, k+3}) + a_3 (\bar{b}_{k+2} c_{3, k+2} + \bar{a}_{k+3} c_{3, k+3}) = 0 \quad \quad (k < r-1).
\]
On the other hand, by \eqref{eqn: truncated c 2n}, we have $c_{2, k+3} = -\frac{\bar{b}_{k+2}}{\bar{a}_{k+3}} c_{2, k+2}$, and hence $\bar{b}_{k+2} c_{3, k+2} + \bar{a}_{k+3} c_{3, k+3} = 0$, which implies $c_{3, k+3} = -\frac{\bar{b}_{k+2}}{\bar{a}_{k+3}} c_{3, k+2}$, $k < r-1$. Now, evidently the recursive situation is exactly the same as that of the proof of Theorem \ref{thm: P kernel and Mz} (more specifically, see \eqref{eqn:alpha mn}). This completes the proof of the theorem.
\end{proof}

As is clear by now, by virtue of Theorem \ref{thm: P kernel and Mz}, the classification criterion of the above theorem is also a classification criterion of tridiagonality of standard Aluthge kernels. Therefore, we have the following:

\begin{Corollary}\label{cor:truncated SK = SK}
If $\clh_k$ is a truncated space, then the Shimorin-Aluthge kernel of $M_z$ is tridiagonal if and only if the standard Aluthge kernel of $M_z$ is tridiagonal.
\end{Corollary}

\section{Final comments and results}\label{sect: final}

First we comment on the assumptions in the definition of truncated kernels (see Section \ref{sec: truncated}). The main advantage of the truncated space corresponding to a truncated kernel is that $F = 0$, where $F$ is the finite rank operator as in \eqref{eqn:F = rank one 2}. In this case, as already pointed out, we have $L_{\tilde{M}_z} = |M_z|^{\h} L_{M_z} |M_z|^{-\h}$. This brings a big cut down in computation. On the other hand, quite curiously, if
\[
b_0 = b_1 = 1 \mbox{~or~} b_0 =1,
\]
and all other $b_i$'s are equal to $0$, then the corresponding standard Aluthge kernel of $M_z$ is tridiagonal kernel but the corresponding Shimorin-Aluthge kernel of $M_z$ is not a tridiagonal kernel. Since computations are rather complicated in the presence of $F$, we only present the result for the following (convincing) case:

\begin{Example}\label{example: truncated SK not SK}
Let $a_n = b_0 = b_1 = 1$ and $b_m = 0$ for all $n \geq 0$ and $m \geq 2$. Let $\clh_k$ denote the tridiagonal space corresponding to the basis $\{(a_n + b_n z) z^n\}_{n\geq 0}$. By \eqref{eqn: Mz matrix} and Theorem \ref{Prop: L Matrix}, we have
\[
[M_z] = \begin{bmatrix}
0 & 0 &  0 & 0 & 0 & \cdots
\\
1 & 0 & 0 & 0 & 0 & \ddots
\\
0 & 1 & 0 & 0 & 0 & \ddots
\\
0 & 1 & 1 & 0 & 0 & \ddots
\\
0 & 0 & 0 & 1 & 0 & \ddots
\\
\vdots & \vdots & \vdots &\vdots &\ddots &\ddots
\end{bmatrix}
\mbox{~~and~~}
[L_{M_z}] = \begin{bmatrix}
0 & 1 &  0 & 0 & 0 & 0 & \cdots
\\
0 & 0 & 1 & 0 & 0 & 0 & \ddots
\\
0 & 0 & -1 & 1 & 0 & 0& \ddots
\\
0 & 0 & 0 & 0 & 1 & 0 & \ddots
\\
\vdots & \vdots & \vdots & \vdots & \vdots &\ddots &\ddots
\end{bmatrix},
\]
respectively. Hence, applying $L_{M_z} L_{M_z}^* = |M_z|^{-2}$ (see Lemma \ref{prop: L tilde T}) to this, we obtain
\[
|M_z|^{-2} =
\begin{bmatrix}
1 & 0 &  0 & 0 & 0 \cdots
\\
0 & 1 & -1 & 0 & 0 & \ddots
\\
0 & -1 & 2 & 0 & 0 & \ddots
\\
0 & 0 & 0 & 1 & 0 & \ddots
\\
\vdots & \vdots & \vdots & \ddots & \ddots
\\
\end{bmatrix}.
\]
Suppose $\alpha=\frac{3+\sqrt{5}}{2}$ and $\beta=\frac{3-\sqrt{5}}{2}$. It is useful to observe that $(1 - \alpha) ( 1 - \beta) +1  = 0$. Set $\begin{bmatrix}a & b \\b & c \end{bmatrix} = \begin{bmatrix}1 & -1 \\-1 & 2 \end{bmatrix}^{\frac{1}{2}}$, where $a = \frac{1}{\sqrt{5}}[\sqrt{\alpha}(1-\beta)-\sqrt{\beta}(1-\alpha)]$ and $b = \frac{1}{\sqrt{5}}[-\sqrt{\alpha}+\sqrt{\beta}]$, and $c = \frac{1}{\sqrt{5}}[-\sqrt{\alpha}(1-\alpha)+\sqrt{\beta}(1-\beta)]$. Clearly
\[
|M_z|^{-1} =
\begin{bmatrix}
1 & 0 & 0 & 0
\\
0 & a & b & 0
\\
0& b & c & 0
\\
0 & 0 & 0 & I
\end{bmatrix}.
\]
From this it follows that $|M_z| f_0 = f_0$, and hence the finite rank operator $F$, as in \eqref{eqn:F = rank one 2}, is given by
\[
F g = \BL g, f_0 \BR_{\clh_k} \Big((M_z^* |M_z| M_z)^{-1} M_z^* |M_z| f_0\Big)= 0 \quad \quad (g \in \clh_k).
\]
Then $F = 0$, and hence \eqref{eqn:tilde L similar L+F} implies that $L_{\tilde{M}_z} = |M_z|^{\h} L_{M_z} |M_z|^{-\h}$. By \eqref{eqn:tilde{X}_{mn}} (also see Step $1$ in the proof of Theorem \ref{thm: truncated}), the coefficient of $z^m \bar{w}^n$ of the Shimorin-Aluthge kernel $\tilde{k}$ is given by $\tilde{X}_{mn} = |M_z|^{\h} P_{\clw}  L_{M_z}^m |M_z|^{-1} L_{M_z}^{*n}|_{\clw}$, $m,n \geq 0$. We compute the coefficient of $z \bar{w}^3$ as
\[
\begin{split}
P_{\clw} L_{M_z} |M_z|^{-1} L_{M_z}^{*3} f_0 & = P_{\clw} L_{M_z} |M_z|^{-1} L_{M_z}^{*2} f_1
\\
& =  P_{\clw} L_{M_z} |M_z|^{-1} L_{M_z}^{*} f_2
\\
& =  P_{\clw} L_{M_z} |M_z|^{-1} (- f_2 + f_3)
\\
& = P_{\clw} L_{M_z} (-b f_1 - c f_2 + f_3)
\\
& = P_{\clw} L_{M_z} (-b f_1)
\\
& = - b f_0.
\end{split}
\]
But $
b = \frac{1}{\sqrt{5}}[-\sqrt{\alpha}+\sqrt{\beta}] \neq 0$, and hence $\tilde{X}_{13} \neq 0$. This implies that the Shimorin-Aluthge kernel is not tridiagonal. On the other hand, the matrix representation of $|M_z|^{-1}$ implies right away that the standard Aluthge kernel is tridiagonal (see Theorem \ref{thm: P kernel and Mz}).
\end{Example}

Now we return to standard Aluthge kernels of shifts (see the definition following Theorem \ref{thm: revisit model}). Let $\clh_k \subseteq \clo(\D)$ be a reproducing kernel Hilbert space. Suppose $M_z$ on $\clh_k$ is left-invertible. Then Theorem \ref{thm: revisit model} says that $\tilde{M}_z$ and $M_z$ on $\clh_{\tilde{k}} (\subseteq \clo(\D))$ are unitarily equivalent, where
\[
\tilde{k}(z, w) := \BL |M_z|^{-1} k(\cdot, w), k(\cdot, z) \BR_{\clh_k} = \Big(|M_z|^{-1} k(\cdot, w)\Big)(z),
\]
for all $z, w \in \D$. In the following, as a direct application of Theorem \ref{thm: P kernel and Mz}, we address the issue of tridiagonal representation of the shift $M_z$ on $\clh_k$.

\begin{Corollary}\label{cor: M TD implies M TD}
In the setting of Theorem \ref{thm: revisit model}, assume in addition that $\cle = \mathbb{C}$ and $\clh_{\tilde{k}}$ is a tridiagonal space with respect to the orthonormal basis $\{f_n\}_{n \geq 0}$, where $f_n(z) = (a_n +b_n z)z^n$, $n\geq 0$. Then $\clh_k$ is a tridiagonal space if and only if
\[
U |M_z| U^* = \begin{bmatrix}
c_{00} & c_{01} & -\frac{\bar{b_1}}{\bar{a_2}}c_{01} & \frac{\bar{b}_1 \bar{b}_2}{\bar{a}_2 \bar{a}_3} c_{01} & \dots
\\
\bar{c}_{01} & c_{11} & c_{12}  &-\frac{\bar{b}_2} {\bar{a}_3}c_{12} & \ddots
\\
-\frac{b_1}{a_2}\bar{c}_{01} & \bar{c}_{12} & c_{22} &c_{23} & \ddots
\\
\frac{b_1b_2}{a_2a_3}\bar{c}_{01} & -\frac{b_2}{a_3}\bar{c}_{12} & \bar{c}_{23} & c_{33} & \ddots
\\
\vdots & \vdots & \vdots & \ddots & \ddots
\end{bmatrix},
\]
with respect to the basis $\{f_n\}_{n\geq 0}$.
\end{Corollary}

\begin{proof} Recall from Theorem \ref{thm: revisit model} that $\clh_{\tilde{k}} = |M_z|^{-\frac{1}{2}} \clh_k$ and $U h = |M_z|^{-\frac{1}{2}} h$, $h \in \clh_k$, defines the intertwining unitary. Set $P := U |M_z| U^*$. Then $P \in \clb(\clh_{\tilde{k}})$ is a positive operator, and for any $z, w \in \D$, we have
\[
\begin{split}
\BL P \tilde{k}(\cdot, w), \tilde{k}(\cdot, z) \BR_{\clh_{\tilde{k}}} & = \BL |M_z| U^* \tilde{k}(\cdot, w), U^* \tilde{k}(\cdot, z) \BR_{\clh_k}
\\
& = \BL |M_z| |M_z|^{- \frac{1}{2}} {k}(\cdot, w), |M_z|^{- \frac{1}{2}} {k}(\cdot, z) \BR_{\clh_k}
\\
& = \BL {k}(\cdot, w), {k}(\cdot, z) \BR_{\clh_k},
\end{split}
\]
as $U \Big(|M_z|^{- \frac{1}{2}} {k}(\cdot, w)\Big) = \tilde k(\cdot, w)$. Hence $k(z, w) = \BL P \tilde{k}(\cdot, w), \tilde{k}(\cdot, z) \BR_{\clh_{\tilde{k}}}$, $z, w \in \D$. The result now follows from Theorem \ref{thm: P kernel and Mz}.
\end{proof}

In particular, if $\tilde{k}$ is a tridiagonal kernel, then for $k$ to be a tridiagonal kernel, it is necessary (as well as sufficient) that
$U |M_z| U^*$ is of the form as in the above statement.

We conclude this paper with the following curious observation which stems from the matrix representations of Shimorin left inverses of shifts on analytic tridiagonal spaces (see Theorem \ref{Prop: L Matrix}). Let $\clh_k$ be an analytic tridiagonal space. Recall that $L_{M_z}$ denotes the Shimorin left inverse of $M_z$. By Lemma \ref{prop: L tilde T}, we have $|M_z|^{-2} = L_{M_z} L_{M_z}^*$. From the matrix representation of $L_{M_z}$ in  Theorem \ref{Prop: L Matrix}, one can check that the matrix representation of $|M_z|^{-2}$ satisfies the conclusion of Theorem \ref{thm: P kernel and Mz}. Consequently, the positive definite scalar kernel
\[
K(z, w) = \langle |M_z|^{-2} k(\cdot, w), k(\cdot, z) \rangle_{\clh_k} \quad \quad (z, w \in \D),
\]
is a tridiagonal kernel. On the other hand, consider
\[
a_n = \begin{cases}
2 & \mbox{if } n = 2 \\
1 & \mbox{otherwise,}
\end{cases} \text{ and }
b_n = \begin{cases}
1 & \mbox{if } n = 2 \\
0 & \mbox{otherwise.}
\end{cases}
\]
Then the shift $M_z$ on the analytic tridiagonal space $\clh_k$ corresponding to the orthonormal basis $\{f_n\}_{n\geq 0}$, where $f_n(z) = (a_n + b_n z) z^n$, $n \geq 0$, is left-invertible. However, a moderate computation reveals that the matrix representation of $|M_z|^{-1}$ does not satisfy the conclusion of Theorem \ref{thm: P kernel and Mz}. In other words, the positive definite scalar kernel
\[
K(z, w) = \langle |M_z|^{-1} k(\cdot, w), k(\cdot, z) \rangle_{\clh_k} \quad \quad (z, w \in \D),
\]
is not a tridiagonal kernel.

\vspace{0.05in}

\noindent\textbf{Acknowledgement:}
The research of the second named author is supported in part by NBHM grant NBHM/R.P.64/2014, and the Mathematical Research Impact Centric Support (MATRICS) grant, File No: MTR/2017/000522 and Core Research Grant, File No: CRG/2019/000908, by the Science and Engineering Research Board (SERB), Department of Science \& Technology (DST), Government of India.

\bibliographystyle{amsplain}

\end{document}